\documentclass[11pt]{article}
\usepackage{geometry}
\usepackage{amsrefs}
\usepackage{amsmath,amsthm,amssymb,amsfonts}
\usepackage{latexsym}
\usepackage{graphicx}
\usepackage{xcolor}
\usepackage{epsfig}
\usepackage{subfig}
\usepackage{pdflscape}
\newcommand{\WSd}{{{\rm WS}_d}}
\newcommand{\WS}{{{\rm WS}_3}}
\newcommand{\spSd}{\spS_d^{d-1}}

\newcommand{\spline}{s}
\newcommand{\objmm}[2]{{#1}^{#2}}
\newcommand{\objme}[2]{\objmm{{\widetilde #1}}{#2}}
\newcommand{\objbmm}[3]{{#1}^{#2,#3}}
\newcommand{\objbme}[3]{\objbmm{{\widetilde #1}}{#2}{#3}}
\newcommand{\Bmm}[1]{\objmm{B}{#1}}

\newcommand{\Hmm}[1]{\objmm{H}{#1}}

\newcommand{\basis}[1]{\mathcal{#1}}
\newcommand{\basisB}{\basis{B}}
\newcommand{\basisH}{\basis{H}}
\newcommand{\basisBmm}[1]{\objmm{\basisB}{#1}}
\newcommand{\basisBme}[1]{\objme{\basisB}{#1}}
\newcommand{\basisHmm}[1]{\objmm{\basisH}{#1}}
\newcommand{\basisHme}[1]{\objme{\basisH}{#1}}
\newcommand{\matR}{\mR}
\newcommand{\matRBmm}[1]{\objbmm{\matR}{\basisB}{#1}}
\newcommand{\matRBme}[1]{\objbme{\matR}{\basisB}{#1}}
\newcommand{\matRHmm}[1]{\objbmm{\matR}{\basisH}{#1}}
\newcommand{\matRHme}[1]{\objbme{\matR}{\basisH}{#1}}
\newcommand{\matC}{\mC}
\newcommand{\matCmm}[1]{\objmm{\matC}{#1}}

\newcommand{\matX}{\mX}
\newcommand{\matXmm}[1]{\objmm{\matX}{#1}}
\newcommand{\matXme}[1]{\objme{\matX}{#1}}
\newcommand{\entry}{{\large $\bullet$}}
\newcommand{\mbar}{{\overline m}}

\newtheorem{theorem}{Theorem}
\newtheorem{lemma}[theorem]{Lemma}
\newtheorem{definition}[theorem]{Definition}

\newtheorem{proposition}[theorem]{Proposition}
\newtheorem{remark}[theorem]{Remark}

\theoremstyle{definition}
\newtheorem{procedure}[theorem]{Procedure}
\newtheorem{examples}[theorem]{Example}

\newcommand{\be}[1]{\begin{equation}\label{#1}}
\newcommand{\ee}{\end{equation}}
\newcommand{\bdef}[1]{\begin{definition}\label{#1}}
\newcommand{\bthm}[1]{\begin{theorem}\label{#1}}
\newcommand{\ethm}{\end{theorem}}
\newcommand{\blem}[1]{\begin{lemma}\label{#1}}
\newcommand{\elem}{\end{lemma}}
\newcommand{\ba}[1]{\begin{array}{#1}}
\newcommand{\ea}{\end{array}}
\newcommand{\bi}{\begin{itemize}}
\newcommand{\ei}{\end{itemize}}
\newcommand{\bex}{\begin{exercise}}
\newcommand{\eex}{\end{exercise}}
\newcommand{\ben}{\begin{enumerate}}
\newcommand{\een}{\end{enumerate}}
\newcommand{\bas}{\begin{align*}}
\newcommand{\eas}{\end{align*}}
\newcommand{\bbm}{\begin{bmatrix}}
\newcommand{\ebm}{\end{bmatrix}}
\newcommand{\bvm}{\begin{vmatrix}}
\newcommand{\evm}{\end{vmatrix}}
\newcommand{\bsm}{\left[\begin{smallmatrix}}
\newcommand{\esm}{\end{smallmatrix}\right]}
\newcommand{\bs}[1]{\begin{slide}{#1}}
\newcommand{\es}{\end{slide}}
\newcommand{\bpf}{\begin{proof}}
\newcommand{\epf}{\end{proof}}

\newcommand{\vn}{{\boldsymbol{n}}}

\newcommand{\vp}{{\boldsymbol{p}}}
\newcommand{\vq}{{\boldsymbol{q}}}

\newcommand{\vt}{{\boldsymbol{t}}}

\newcommand{\vw}{{\boldsymbol{w}}}
\newcommand{\vx}{{\boldsymbol{x}}}
\newcommand{\vy}{{\boldsymbol{y}}}

\newcommand{\spP}{{\mathbb{P}}}

\newcommand{\spS}{{\mathbb{S}}}

\newcommand{\cT}{{\mathcal{T}}}

\newcommand{\NN}{{\mathbb{N}}}

\newcommand{\RR}{{\mathbb{R}}}

\newcommand{\abs}[1]{\lvert#1\rvert}

\newcommand{\makehosquare}[3]%
{\dimen0=#1\advance\dimen0 by -#3%
\vrule height#1 width#2 depth0pt \kern-#2%
\vrule height#1 width#1 depth-\dimen0 \kern-#1%
\vrule height#2 width#1 depth0pt \kern-#3%
\vrule height#1 width#3 depth0pt%
}

\newcommand{\mC}{{\boldsymbol{C}}}

\newcommand{\mI}{{\boldsymbol{I}}}

\newcommand{\mO}{{\boldsymbol{O}}}

\newcommand{\mR}{{\boldsymbol{R}}}

\newcommand{\mX}{{\boldsymbol{X}}}

\begin{document}
\captionsetup[subfigure]{labelformat=empty}

\title{\textbf{A parsimonious approach to $C^2$ cubic splines on \\ arbitrary triangulations: Reduced macro-elements \\ on the cubic Wang--Shi split}}

\author{Tom Lyche\footnote{Tom Lyche, Dept. of Mathematics, University of Oslo, Norway, \textit{email: tom@math.uio.no}}
,
Carla Manni\footnote{Carla Manni, Dept. of Mathematics, University of Rome Tor Vergata, Italy, \textit{email: manni@mat.uniroma2.it}}
,
Hendrik Speleers\footnote{Hendrik Speleers, Dept. of Mathematics, University of Rome Tor Vergata, Italy,
\textit{email: speleers@mat.uniroma2.it}}}

\date{}

\maketitle

\begin{abstract}
We present a general method to obtain interesting subspaces of the $C^2$ cubic spline space defined on the cubic Wang--Shi refinement of a given arbitrary triangulation $\cT$.
These subspaces are characterized by specific Hermite degrees of freedom associated with only the vertices and edges of $\cT$, or even only the vertices of $\cT$.
Each subspace still contains cubic polynomials while saving a consistent number of degrees of freedom compared with the full space. The dimension of the considered subspaces can be as small as six times the number of vertices of $\cT$.
The method fits in the setting of macro-elements: any function of such a subspace can be constructed on each triangle of $\cT$ separately by specifying the necessary Hermite degrees of freedom. The explicit local representation in terms of a local simplex spline basis is also provided. This simplex spline basis intrinsically takes care of the complex geometry of the Wang--Shi split, making it transparent to the user.

\end{abstract}

\section{Introduction}
Smooth splines on triangulations are a popular tool to approximate or interpolate function data \cite{Lai.Schumaker07}.
Following the classical finite element approach \cite{Ciarlet.78}, such spline spaces are often characterized by means of Hermite degrees of freedom associated with the vertices, edges, and/or triangles of the triangulation.
In practical situations, however, data values are not always available in the interior of the triangles and are often only attached to the vertices of a given triangulation. Therefore, it is preferable to avoid Hermite degrees of freedom associated with the triangles, and even better, to have them only associated with the vertices. Additionally, for the sake of approximation quality, the provided construction should ensure reproduction of polynomials up to the highest possible degree. Finally, local constructions that can be performed on each triangle separately are highly attractive from a computational viewpoint.

In classical finite element constructions, it is a common procedure to remove Hermite degrees of freedom associated with the triangles or edges by expressing them as a linear combination, with suitable coefficients, of the degrees of freedom associated with the vertices \cite{Ciarlet.78}. However, the full approximation order is not always preserved in those constructions; this is, e.g., the case for the Bell element \cite{Bell.69}. Expressing degrees of freedom in terms of other degrees of freedom can also be helpful in the construction of B-spline-like basis functions for reduced spline spaces on triangulations \cites{Speleers10,Speleers13}.

In the univariate case, splines of maximal smoothness have proven to be very useful. They maintain the same approximation order as less smooth splines of the same degree but involve fewer degrees of freedom and have less tendency to oscillate \cites{Sande20,Schumaker07}. 
When moving to the bivariate setting on triangulations, maximal smoothness is still very appealing but becomes an arduous task to achieve \cites{Lai.Schumaker07,Schumaker15}. Bivariate spline spaces with too low degree compared to the smoothness are exposed to several shortcomings: they may lack a stable dimension, optimal approximation power, and stable locally supported bases.

To construct $C^2$ cubic spline functions on a given triangulation $\cT$, it is opportune to refine each triangle of $\cT$ by the Wang--Shi (WS) split of order three; see \cites{LycheMS22,Wang01,Wang.90}. Even though the split is rather complex, the corresponding $C^2$ cubic spline space possesses a stable dimension and can be simply characterized by means of certain Hermite degrees of freedom, of which there are six per vertex, three per edge, and one per triangle in the triangulation. These degrees of freedom allow us to locally build a spline function on each triangle of $\cT$ separately. Furthermore, the space is equipped with a local simplex spline basis \cite{LycheMS22}, which enjoys properties similar to the Bernstein polynomial basis on a triangle. Thanks to the characteristics of this local simplex spline basis, one can avoid to consider separate polynomial representations on each of the polygonal subregions of the WS split. Instead, there is a single control net to facilitate control and early visualization of a spline function over each triangle of $\cT$. This makes that the complex geometry of the WS split is transparent to the user. A similar simplex spline construction for $C^3$ quartic splines has been developed in \cite{LycheMS23}.
  
In this paper, we present a general method to obtain interesting subspaces of the $C^2$ cubic spline space defined on a WS refined triangulation. These subspaces are characterized by specific Hermite degrees of freedom associated with only the vertices and edges of $\cT$, or even only the vertices of $\cT$. Each subspace contains cubic polynomials, a necessary condition for attaining full approximation order. Moreover, any element of such a subspace can be constructed on each triangle of $\cT$ separately by specifying the necessary Hermite degrees of freedom. We also provide the explicit local representation of such an element in terms of the local simplex spline basis. Therefore, we provide a construction of $C^2$ cubic splines on triangulations that is parsimonious in the sense that it produces spaces still containing cubic polynomials but requiring fewer degrees of freedom. The dimension of the considered subspaces can be as small as six times the number of vertices of the given triangulation.

The remainder of the paper is divided into four sections. Section~\ref{sec:preliminaries} collects some introductory material about the local $C^2$ cubic spline space on a single WS refined triangle and the global $C^2$ cubic spline space on a WS refined triangulation. Section~\ref{sec:reduced-spaces-local} describes a procedure how to peel away Hermite degrees of freedom from the local spline space on each triangle and how to represent the resulting subspaces in terms of a local simplex spline basis. These local subspaces can then be assembled to form a global subspace containing cubic polynomials as outlined in Section~\ref{sec:reduced-spaces-global}. Section~\ref{sec:conclusion} concludes with some final remarks. 
Since this paper is a follow-up to \cite{LycheMS22}, we use the same notation and terminology as introduced in \cite{LycheMS22} for ease of comprehension.

\section{The cubic spaces $\spS_3^2(\Delta_{\WS})$ and $\spS_3^2(\cT_{\WS})$}
\label{sec:preliminaries}
In this section, we provide some preliminary material necessary for the rest of the paper. We first summarize the construction of $C^2$ cubic splines on the WS split of order three of a given triangle by specifying suitable Hermite degrees of freedom \cites{LycheMS22,Wang.90}, and then extend it to obtain $C^2$ cubic splines on an arbitrary triangulation $\cT$. Subsequently, we recall from \cites{LycheMS22} the definition and main properties of a local simplex spline basis for the considered local space on any triangle of $\cT$, and we describe its relation with the above mentioned Hermite degrees of freedom.

\begin{figure}[t!]
\centering
\includegraphics[width=3.8cm]{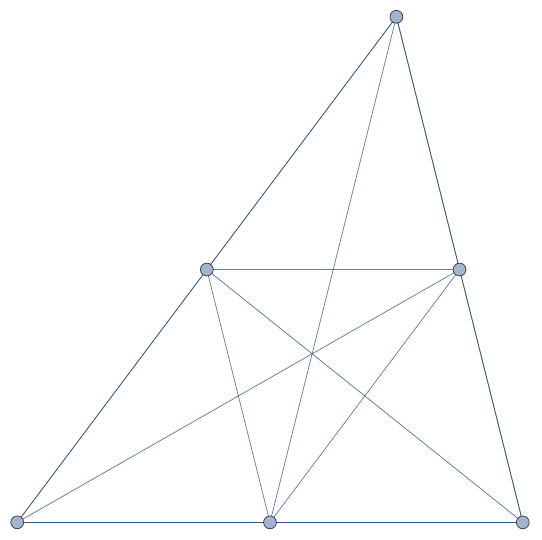} \
\includegraphics[width=3.8cm]{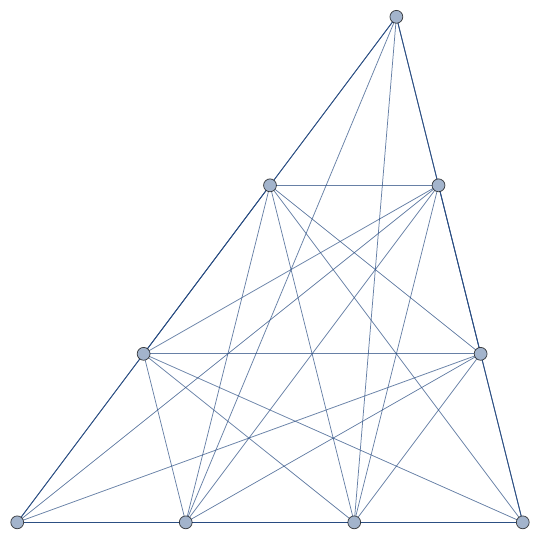} \
\includegraphics[width=3.8cm]{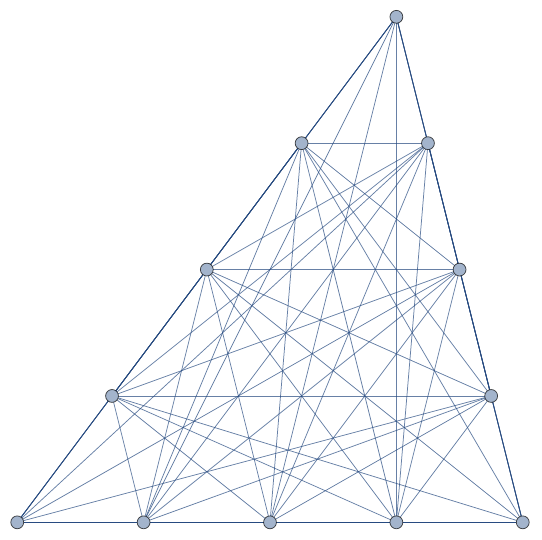} \\[0.2cm]
\caption{$\WSd$ splits for $d=2,3,4$.}
\label{fig:splits}
\end{figure}

\subsection{Definition of the spaces}
Let $\vp_1,\vp_2,\vp_3$ be three noncollinear points in $\RR^2$ and consider the triangle $\Delta:=\langle\vp_1,\vp_2,\vp_3\rangle$. We focus on a special refinement of $\Delta$.
Given a degree $ d\in\NN$, we divide each edge of $\Delta$ into $d$ equal segments, respectively, resulting into $3d$ boundary points. Then, we partition $\Delta$ into a number of regions delineated by the complete graph connecting those boundary points. This is called the $\WSd$ split of $\Delta$ as it was originally proposed by Wang and Shi \cite{Wang.90}.
We denote the obtained mesh structure by $\Delta_{\WSd}$.
The cases $d=2,3,4$ are shown in Figure~\ref{fig:splits}. 
For $d=2$ we obtain the well-known PS-12 split \cite{Powell.Sabin77}.

Let $\spP_d$ be the space of bivariate polynomials, with real coefficients, of total degree at most $d$.
We define the spline space
$$
\spSd(\Delta_{\WSd}):=\{\spline\in C^{d-1}(\Delta): \spline_{|\tau} \in\spP_d,\   \tau\text{ is polygon of } \Delta_{\WSd}\}.
$$
In this paper, we focus on the case $d=3$. From \cite[Theorem~2]{LycheMS22} we know that
$$ 
\dim(\spS_3^2(\Delta_{\WS}))=\dim\spP_3+18=28,
$$
where $18$ is the number of interior lines in the complete graph of $\WS$.

\begin{figure}[t!]
\centering
\begin{picture}(130,160)(0,-5)
\put(0,10){\includegraphics[width=4.75cm]{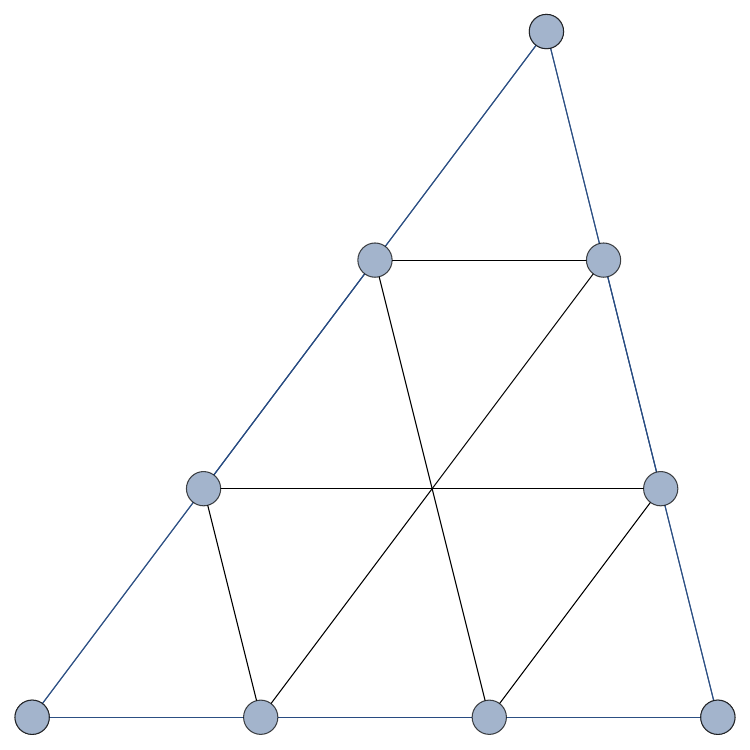}}
\put(0,0){$\vp_1$}
\put(125,0){$\vp_2$}
\put(95,150){$\vp_3$}
\put(40,0){$\vp_{3,1}$}
\put(81,0){$\vp_{3,2}$}
\put(41,97){$\vp_{2,3}$}
\put(10,55){$\vp_{2,1}$}
\put(129,55){$\vp_{1,2}$}
\put(119,97){$\vp_{1,3}$}
\end{picture}
\caption{Labeling of the knots on the boundary of the triangle $\Delta$.}
\label{fig:bvertexorder}
\end{figure}
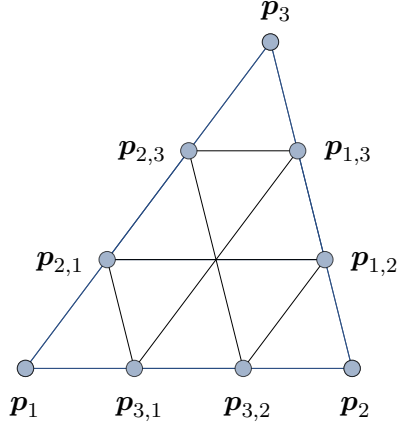

We now introduce the points
(see Figure~\ref{fig:bvertexorder})
\begin{equation}\label{eq:knot-points}
\begin{alignedat}{4}
\vp_{1,2}&:=\frac{2}{3}\vp_2+\frac{1}{3}\vp_3, &\quad
\vp_{1,3}&:=\frac{1}{3}\vp_2+\frac{2}{3}\vp_3, \\
\vp_{2,1}&:=\frac{2}{3}\vp_1+\frac{1}{3}\vp_3, &
\vp_{2,3}&:=\frac{1}{3}\vp_1+\frac{2}{3}\vp_3, \\
\vp_{3,1}&:=\frac{2}{3}\vp_1+\frac{1}{3}\vp_2,&
\vp_{3,2}&:=\frac{1}{3}\vp_1+\frac{2}{3}\vp_2,
\end{alignedat}
\end{equation}
which play the role of knots of the spline basis defined later (see Section~\ref{sec:basis}), 
and the points
\begin{equation} \label{eq:dof-points-1}
\begin{alignedat}{4}
\vq_1 &:=\frac{1}{2}\vp_2+\frac{1}{2}\vp_3, \quad
\vq_2 :=\frac{1}{2}\vp_1+\frac{1}{2}\vp_3, \quad
\vq_3 :=\frac{1}{2}\vp_1+\frac{1}{2}\vp_2, \\
\vq &:=\frac{1}{3}\vp_1+\frac{1}{3}\vp_2+\frac{1}{3}\vp_3.\\
\end{alignedat}
\end{equation}
Then we have the following characterization. For its proof, we refer the reader to \cite[Corollary~1]{LycheMS22} and a schematic visualization can be found in Figure~\ref{fig:dof}.
\begin{proposition} \label{pro:hermite}
For given data $f_{k,\alpha,\beta}$, $g_k$, $g_{k,l}$, and $h_0$, there exists a unique spline $\spline\in\spS_3^2(\Delta_{\WS})$ such that
\begin{align*}
D_x^\alpha D_y^\beta\spline(\vp_k) &= f_{k,\alpha,\beta}, \quad 0\leq \alpha+\beta\leq 2,\quad k=1,2,3, \\
D_{\vn_k}\spline(\vq_{k}) &= g_{k},\quad  D_{\vn_k}^2\spline(\vp_{k,l}) = g_{k,l},\quad  k,l=1,2,3, \quad k\neq l, \\
\spline(\vq) &= h_0,
\end{align*}
where $\vn_k$ is the normal direction of the edge opposite to vertex $\vp_k$, and the points $\vp_{k,l}$, $\vq_{k}$, and $\vq$ are defined in \eqref{eq:knot-points} and \eqref{eq:dof-points-1}.
\end{proposition}

\begin{figure}[t!]
\centering
\includegraphics[trim=60 30 60 20,clip,width=7cm]{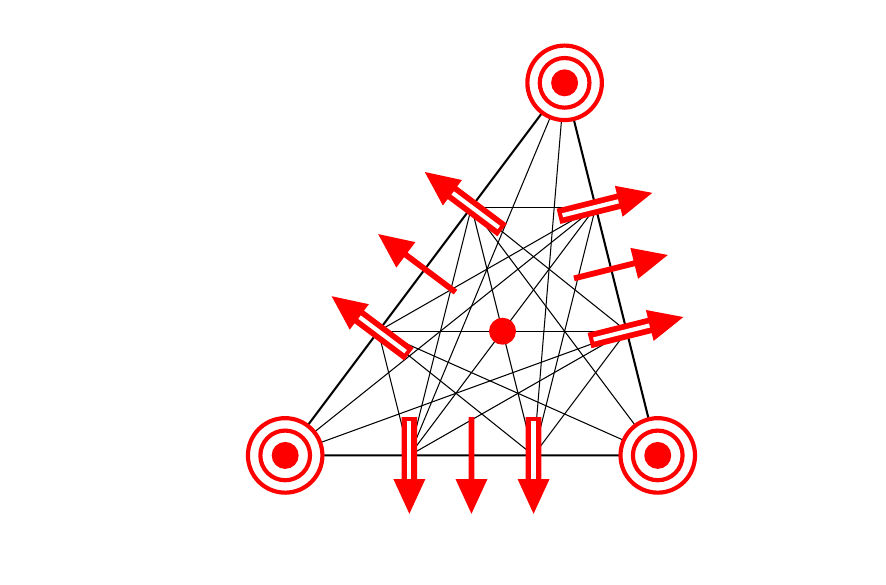}
\caption{Hermite degrees of freedom on the $\WS$ split.}
\label{fig:dof}
\end{figure}

Let $\cT$ be a triangulation of a polygonal domain $\Omega\subset\RR^2$ and let $\cT_{\WS}$ denote its refinement obtained by taking the $\WS$ split of each of its triangles. In this paper, we delve into the $C^2$ cubic spline space on $\cT_{\WS}$, i.e.,
$$
\spS_3^2(\cT_{\WS}):=\{\spline\in C^2(\Omega): \spline_{|\tau}\in \spP_3, \ \tau \text{ is polygon of } \cT_{\WS} \}.
$$
From Proposition~\ref{pro:hermite} we directly deduce that (see also \cites{LycheMS22,Wang.90})
\begin{equation}
\label{eq:dim-space}
\dim(\spS_3^2(\cT_{\WS}))=6n_V+3n_E+n_T,
\end{equation}
where $n_V$, $n_E$, $n_T$ are the number of vertices, edges, and triangles of $\cT$, respectively.
Moreover, any spline function of $\spS_3^2(\cT_{\WS})$ can be locally constructed on each (macro-)triangle $\Delta$ of $\cT$ via the Hermite data in Proposition~\ref{pro:hermite}. However, the complex structure of the $\WS$ split, which consists of 75 polygonal regions (including triangles, quadrilateral, and pentagons), makes it difficult to handle a plain piecewise polynomial representation of such a spline. For this reason, in \cite{LycheMS22}, a local basis has been provided for $\spS_3^2(\Delta_{\WS})$ in terms of (scaled) simplex splines.
We summarize the properties of such a local basis in the next section.

\subsection{Local simplex spline basis for $\spS_3^2(\Delta_{\WS})$}
\label{sec:basis}
\begin{figure}[t!]
\centering
\subfloat[1]{\includegraphics[width=1.5cm]{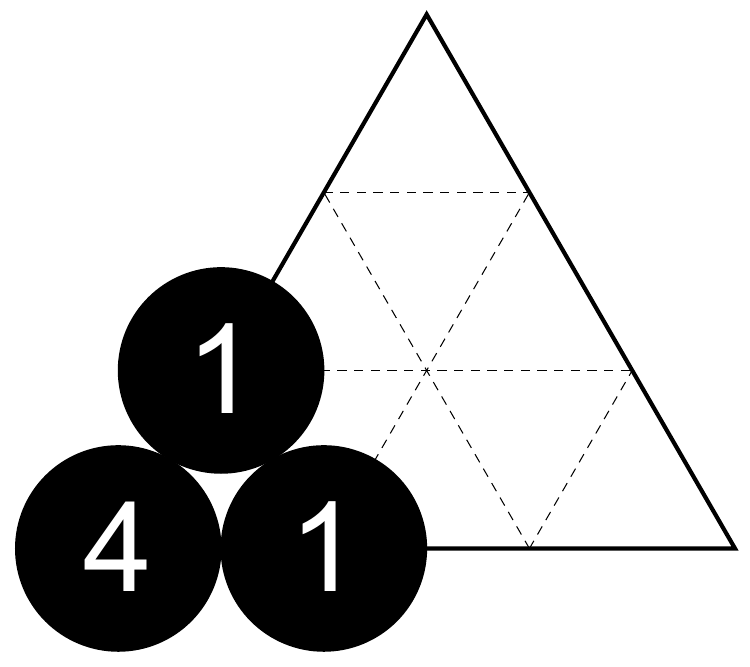}} \quad
\subfloat[2]{\includegraphics[width=1.5cm]{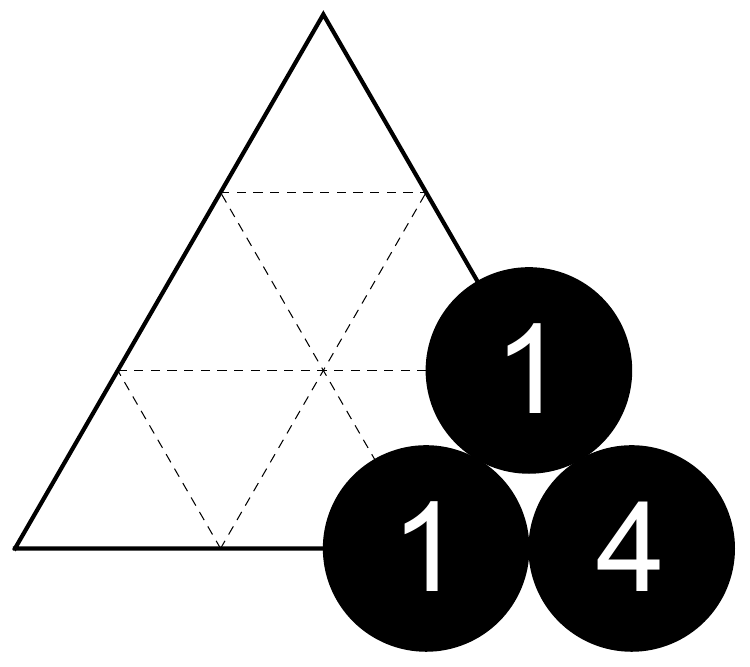}} \quad
\subfloat[3]{\includegraphics[width=1.3cm]{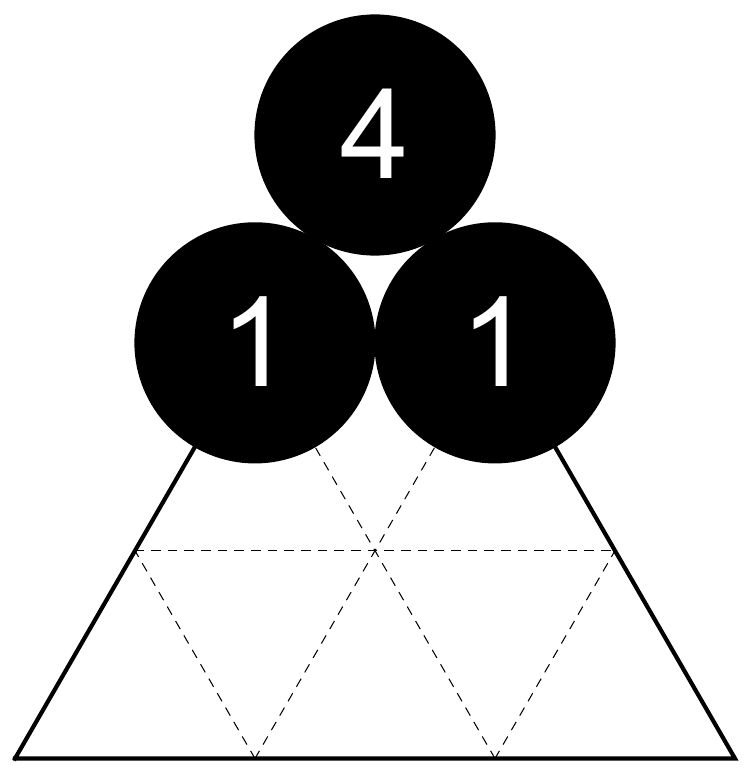}} \quad
\subfloat[4]{\includegraphics[width=1.5cm]{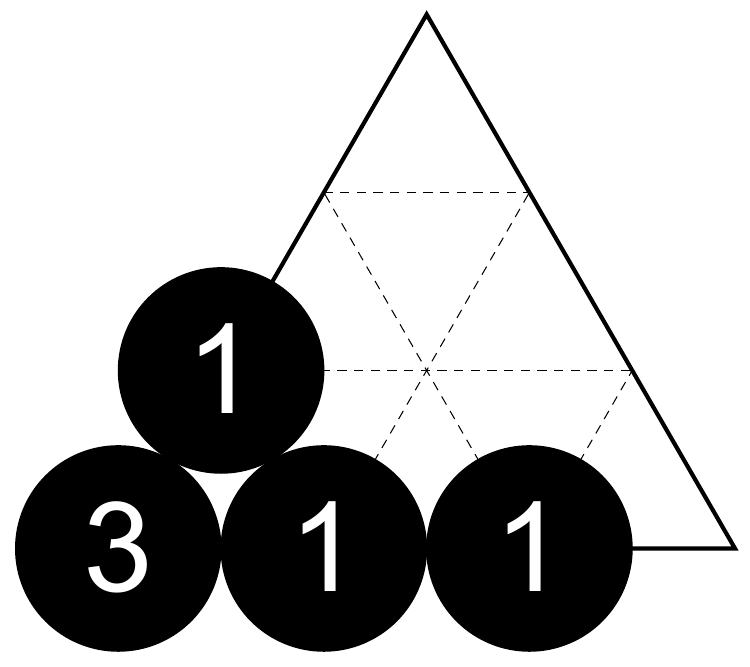}} \quad
\subfloat[5]{\includegraphics[width=1.5cm]{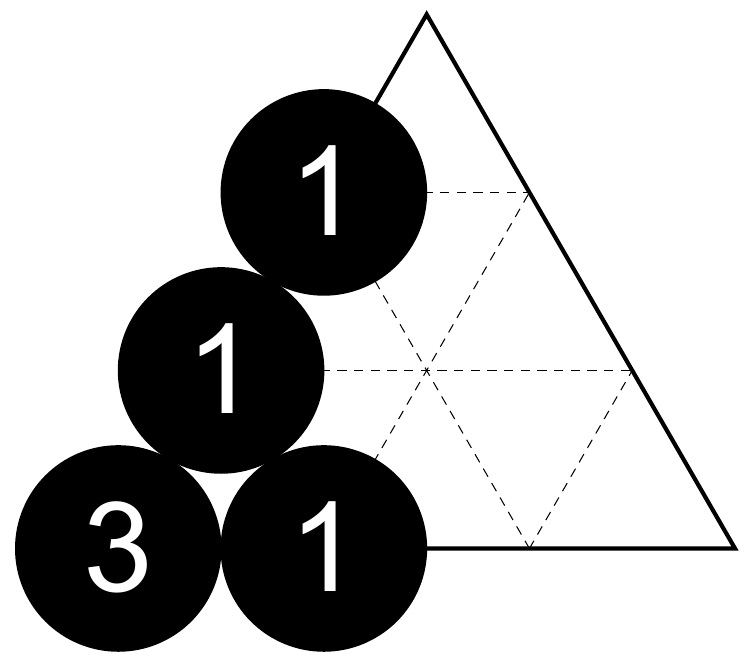}} \quad
\subfloat[6]{\includegraphics[width=1.5cm]{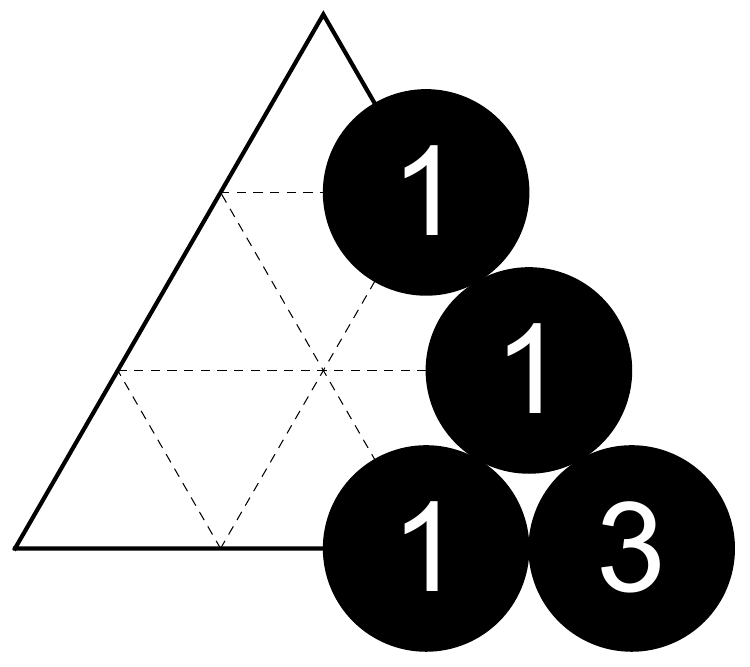}} \quad
\subfloat[7]{\includegraphics[width=1.5cm]{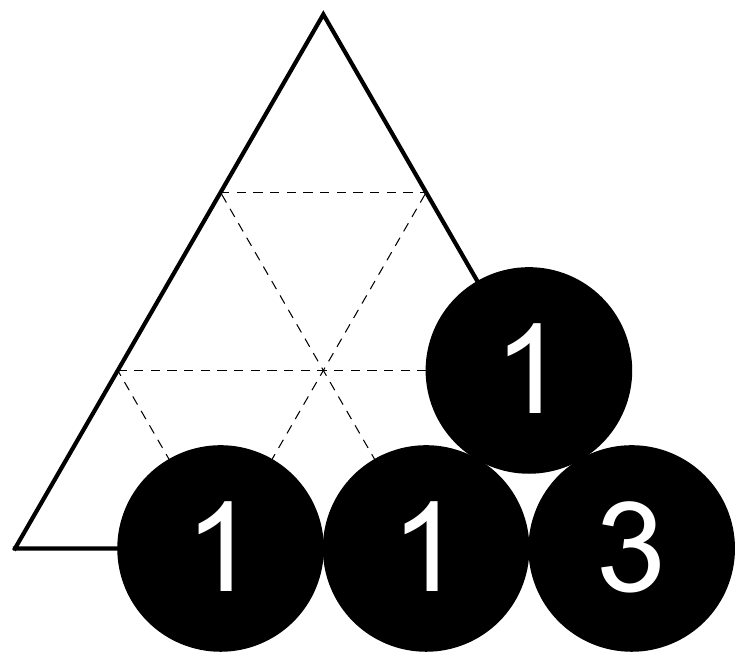}} \hspace*{-0.7cm}
\\
\subfloat[8]{\includegraphics[width=1.35cm]{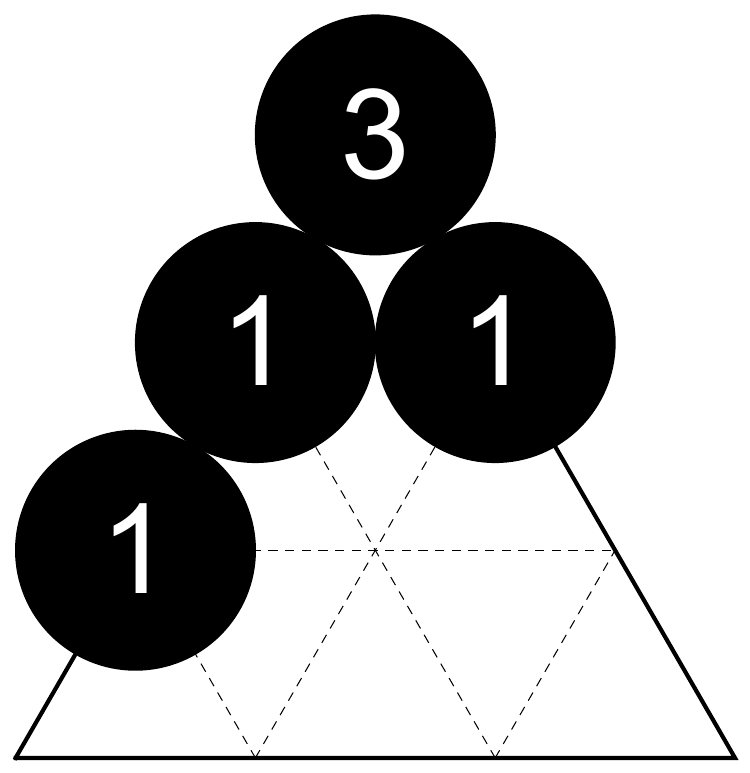}} \quad
\subfloat[9]{\includegraphics[width=1.35cm]{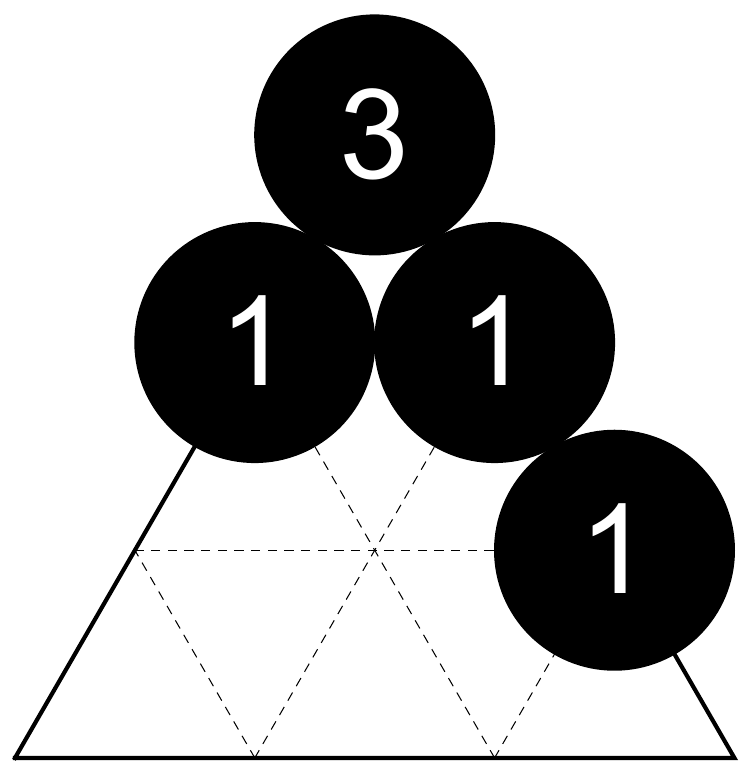}} \quad
\subfloat[10]{\includegraphics[width=1.65cm]{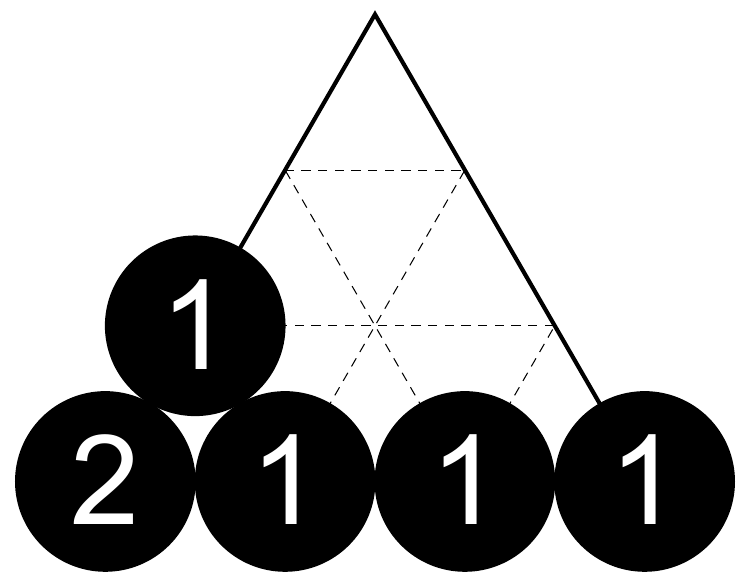}} \quad
\subfloat[11]{\includegraphics[width=1.5cm]{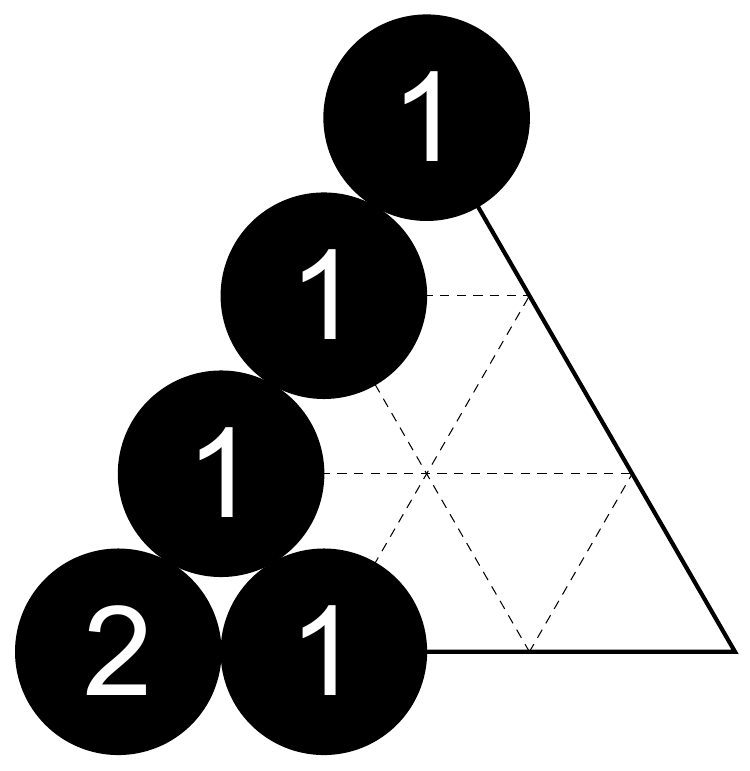}} \quad
\subfloat[12]{\includegraphics[width=1.5cm]{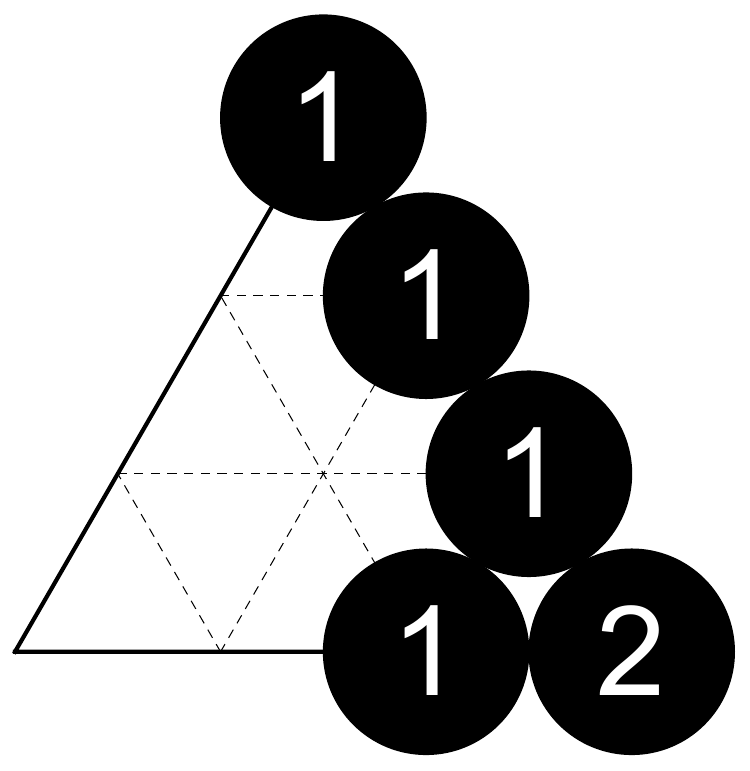}} \quad
\subfloat[13]{\includegraphics[width=1.65cm]{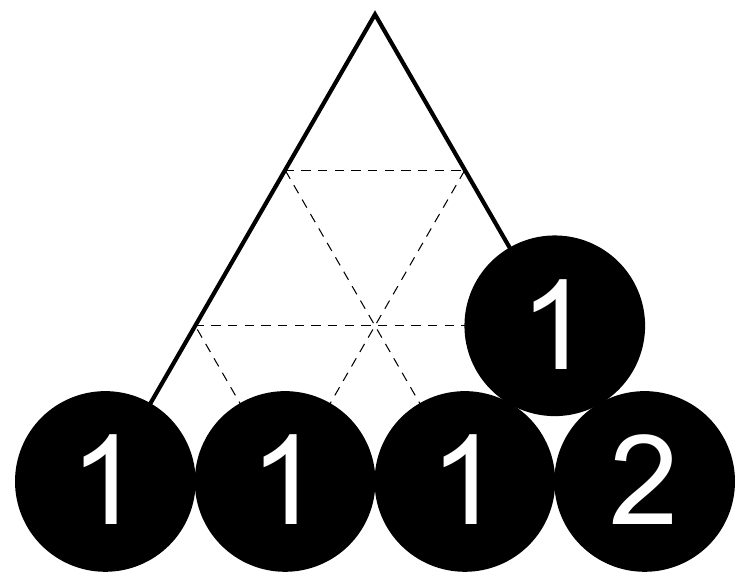}} \quad
\subfloat[14]{\includegraphics[width=1.5cm]{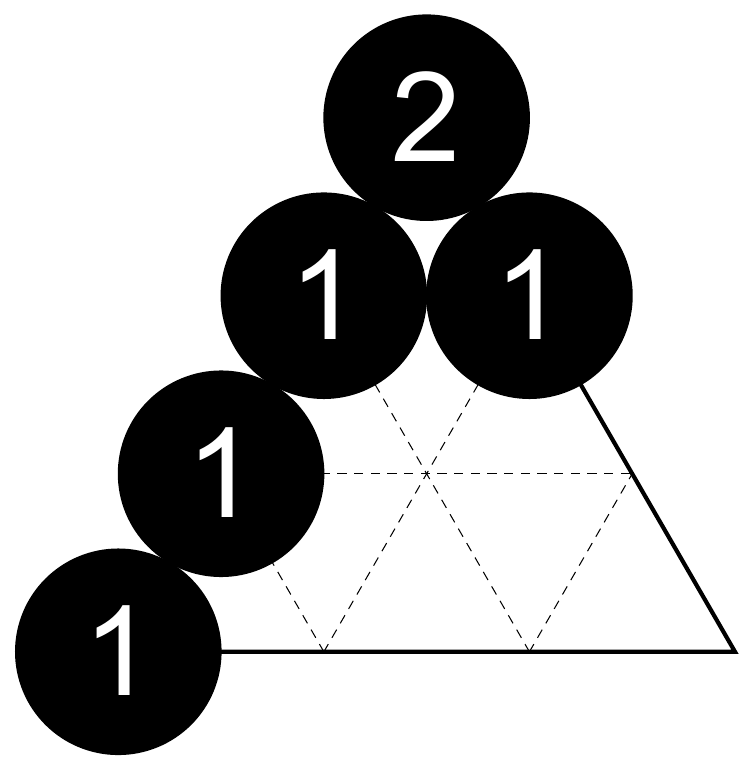}} \hspace*{-0.7cm}
\\
\subfloat[15]{\includegraphics[width=1.5cm]{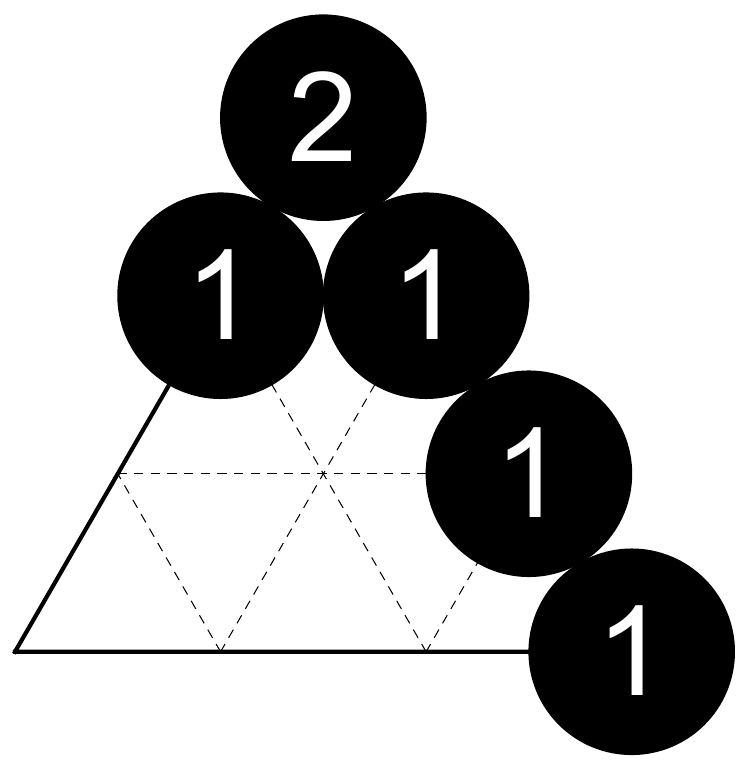}} \quad
\subfloat[16]{\includegraphics[width=1.5cm]{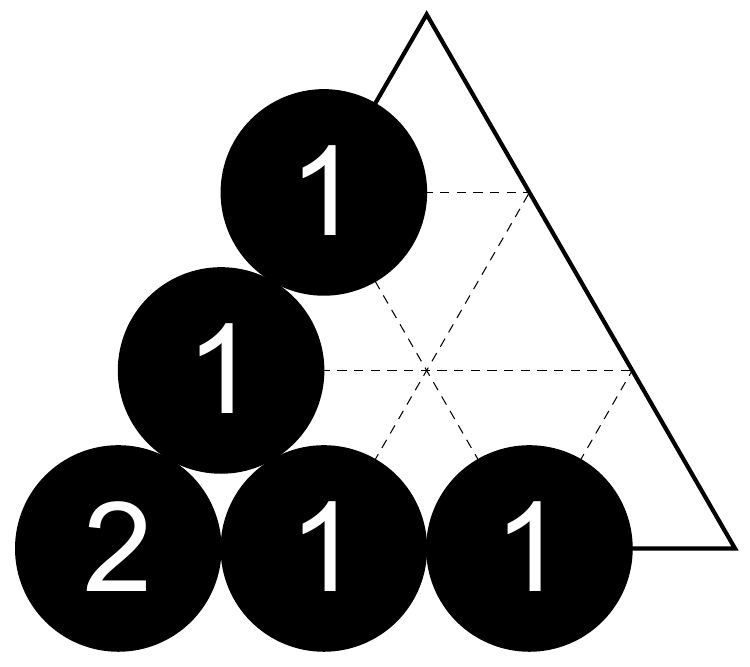}} \quad
\subfloat[17]{\includegraphics[width=1.5cm]{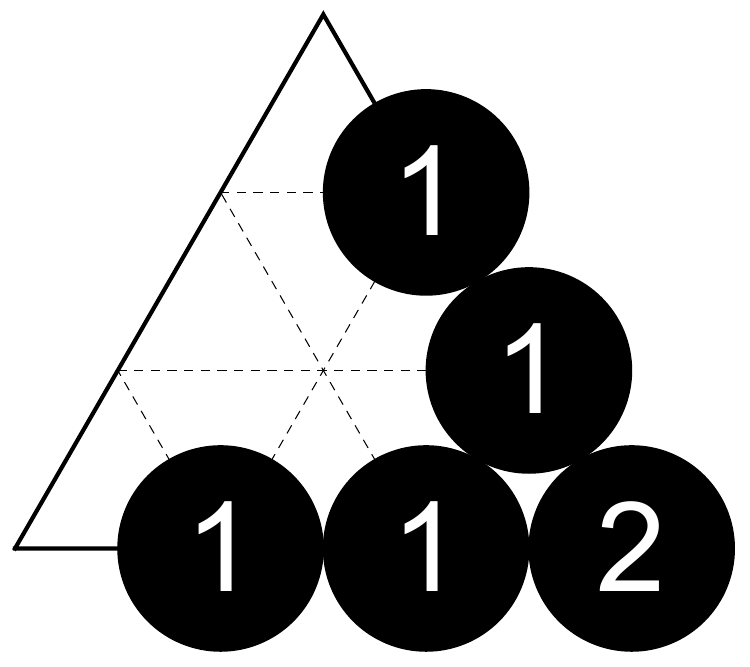}} \quad
\subfloat[18]{\includegraphics[width=1.3cm]{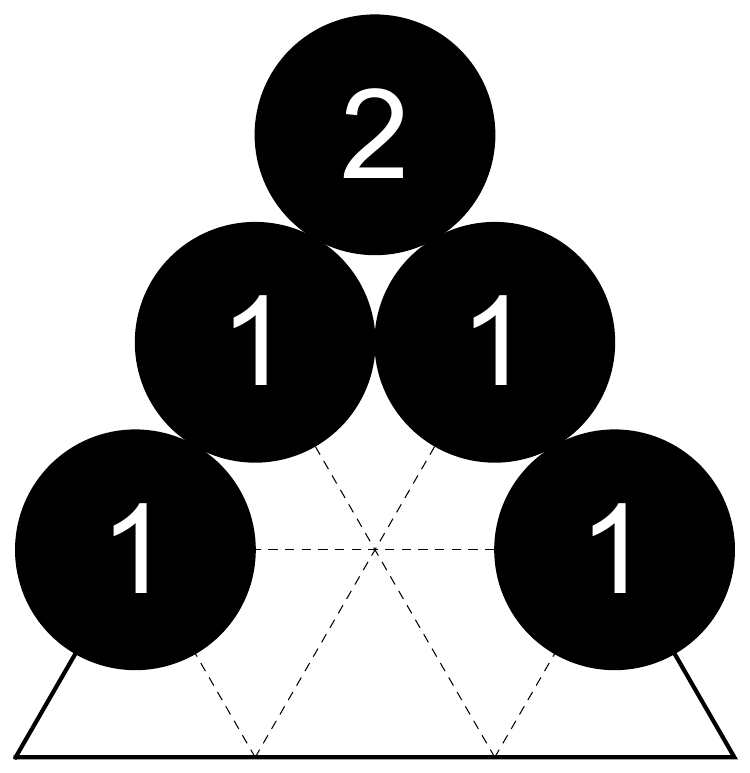}} \quad
\subfloat[19]{\includegraphics[width=1.65cm]{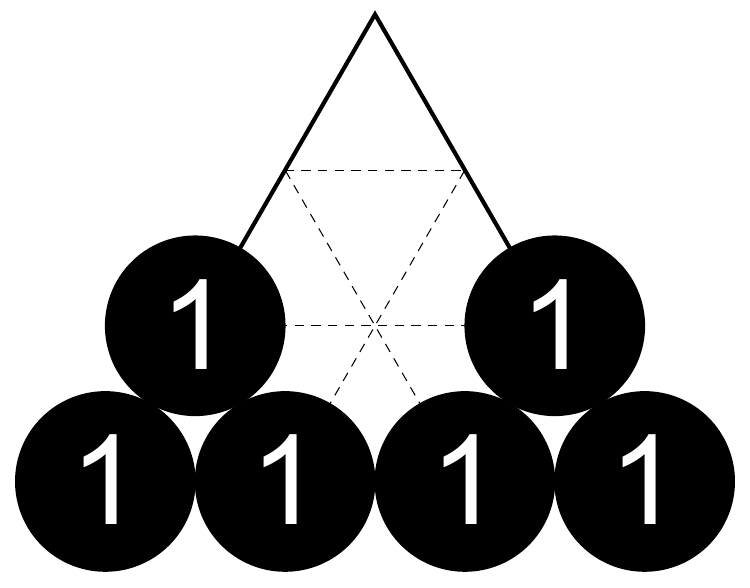}} \quad
\subfloat[20]{\includegraphics[width=1.5cm]{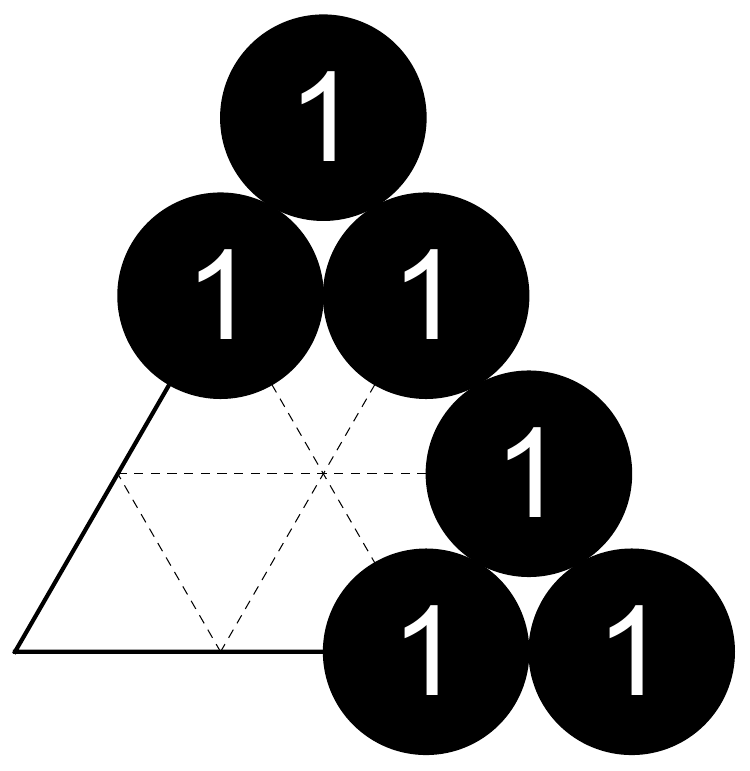}} \quad
\subfloat[21]{\includegraphics[width=1.5cm]{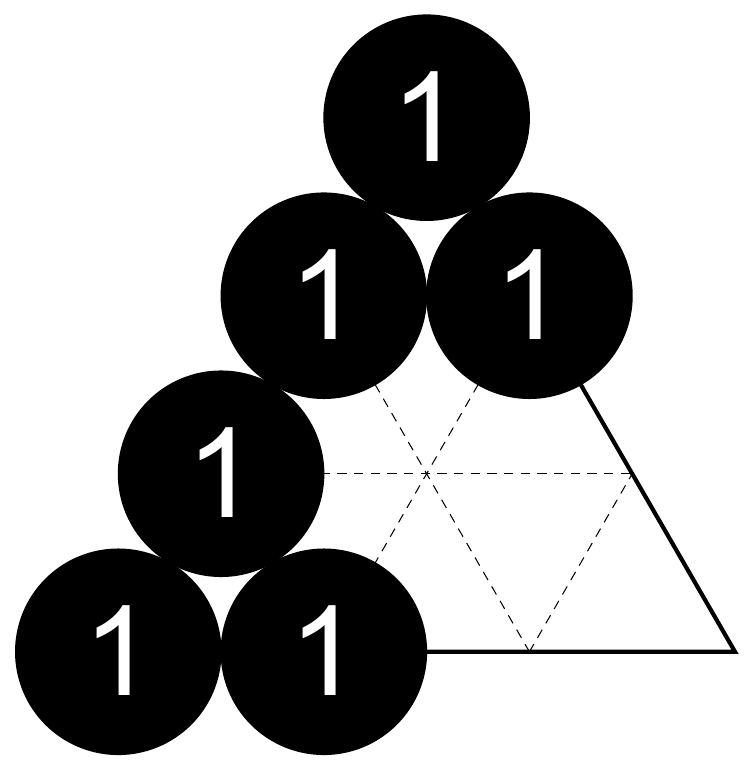}} \hspace*{-0.7cm}
\\
\subfloat[22]{\includegraphics[width=1.5cm]{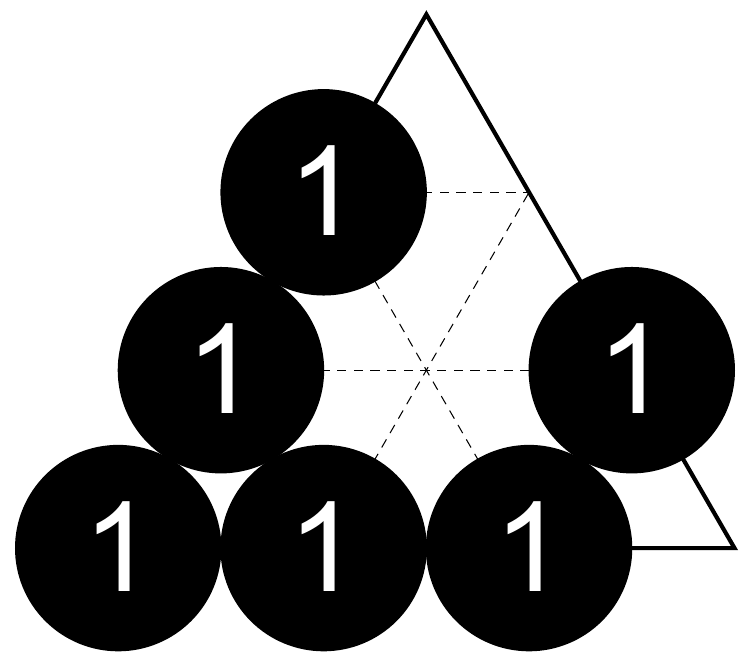}} \quad
\subfloat[23]{\includegraphics[width=1.5cm]{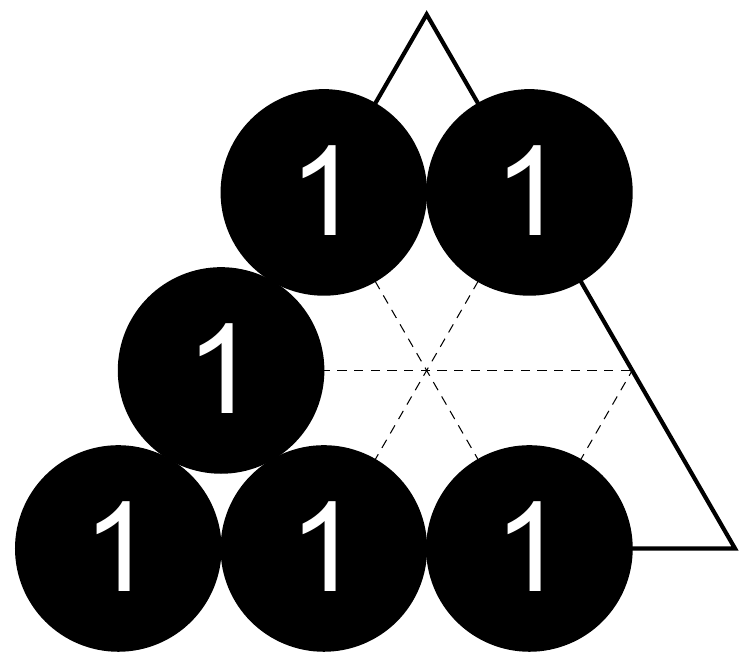}} \quad
\subfloat[24]{\includegraphics[width=1.5cm]{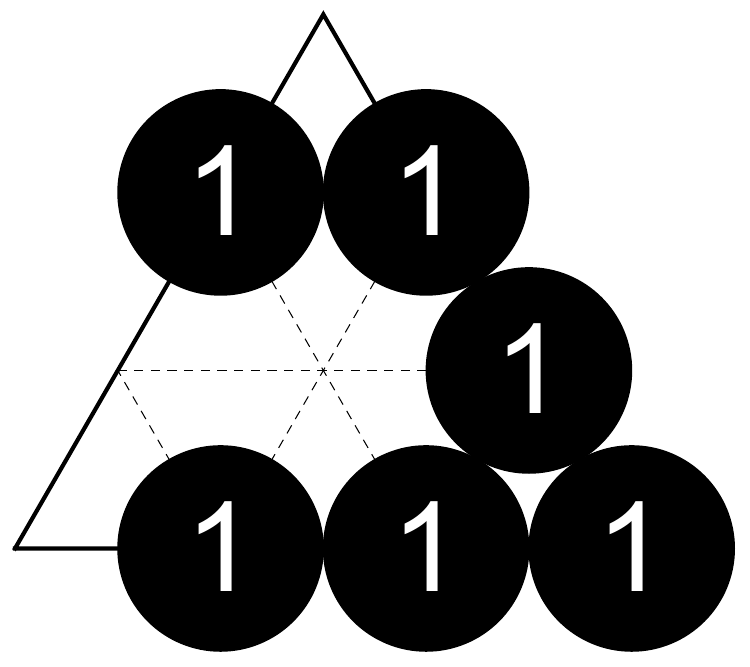}} \quad
\subfloat[25]{\includegraphics[width=1.5cm]{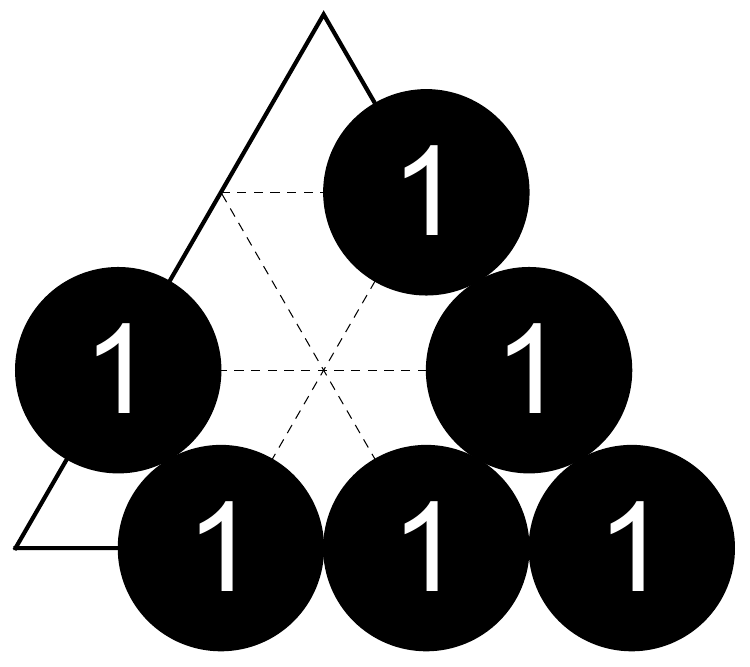}} \quad
\subfloat[26]{\includegraphics[width=1.3cm]{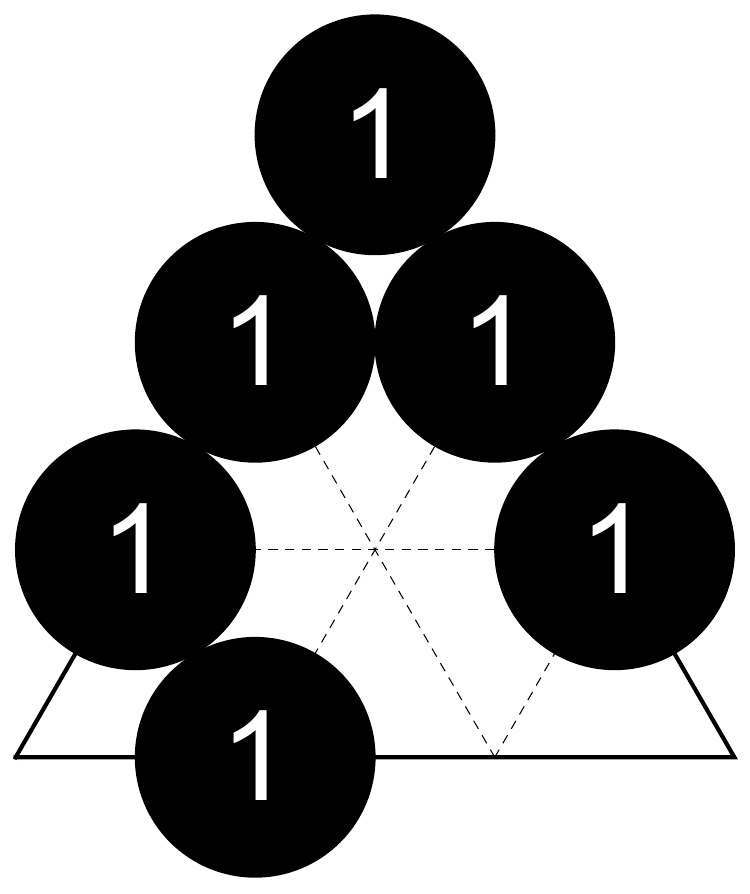}} \quad
\subfloat[27]{\includegraphics[width=1.3cm]{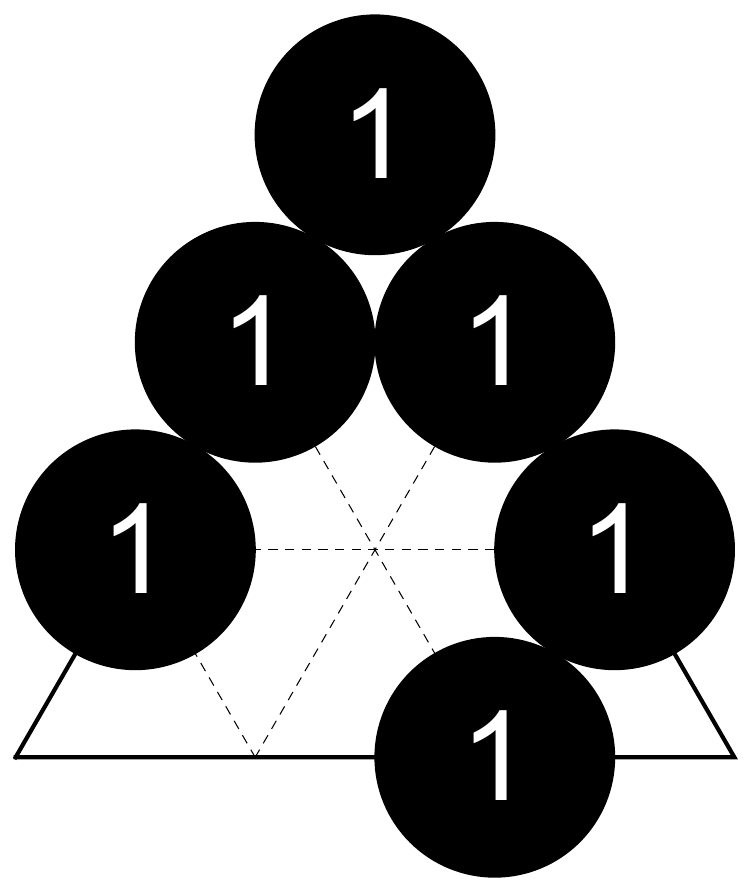}} \quad
\subfloat[28]{\includegraphics[width=1.3cm]{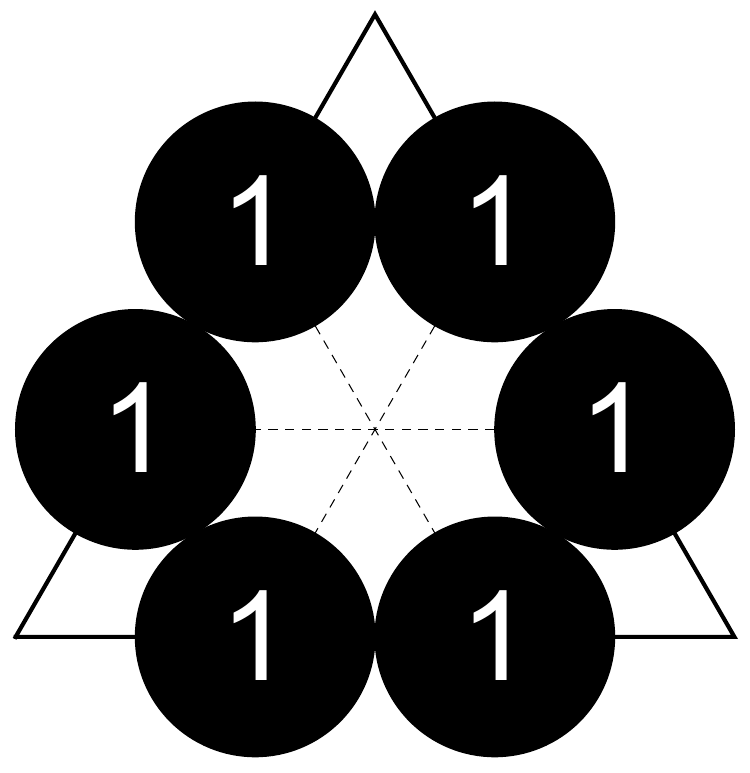}} \hspace*{-0.7cm}
\caption{Knot configurations of the considered simplex spline basis functions for $\spS_3^2(\Delta_{\WS})$. Each black disc shows the position of a knot and the number inside indicates its multiplicity.}
\label{fig:allknots}
\end{figure}
For a given triangle $\Delta:=\langle\vp_1,\vp_2,\vp_3\rangle$, we consider the $28$ cubic simplex splines schematically illustrated in Figure~\ref{fig:allknots}, where each simplex spline has six (including multiplicity) knots chosen among the three vertices $\vp_1,\vp_2,\vp_3$ and the six points in \eqref{eq:knot-points}. The lines in the complete graph of the set of knots are called knot lines. 
According to the properties of simplex splines (see \cites{Micchelli79,Prautsch.Boehm.Paluszny02} for more details), each of the considered functions 
\begin{itemize}
\item is supported on the convex hull of its knots;
\item is a polynomial of degree $d=3$ on each region of the partition $\Delta_\WS$;
\item is of class $C^{4-\mu}$ across a knot line, where $\mu$ is the number of knots on that knot line, including multiplicities.
\end{itemize} 
After multiplying the above mentioned simplex splines by the scaling factors
$$
{\small
\vw:=\frac{\abs{\Delta}}{15}\biggl[\frac{1}{6},\frac{1}{6},
  \frac{1}{6},\frac{1}{3},\frac{1}{3},\frac{1}{3},\frac{1}{3},
  \frac{1}{3},\frac{1}{3},\frac{1}{2},\frac{1}{2},\frac{1}{2},
  \frac{1}{2},\frac{1}{2},\frac{1}{2},\frac{2}{3},\frac{2}{3},
  \frac{2}{3},\frac{5}{6},\frac{5}{6},\frac{5}{6},\frac{2}{3},
  \frac{2}{3},\frac{2}{3},\frac{2}{3},\frac{2}{3},\frac{2}{3},1
\biggr],}
$$
where $|\Delta|$ denotes the area of $\Delta$, 
we obtain the vector of $28$ functions 
\begin{equation}
\label{eq:basisB}
\basisB := [B_1,\ldots,B_{28}],
\end{equation}
which forms a nonnegative partition-of-unity basis
for the space $\spS_3^2(\Delta_{\WS})$; see \cite[Theorem~3]{LycheMS22}. Some of the basis functions are depicted in Figures~\ref{fig:B1-B4-B10}--\ref{fig:B28}. The remaining ones can be obtained by symmetry.

\begin{figure}[t!]
\centering
\includegraphics[width=3.8cm]{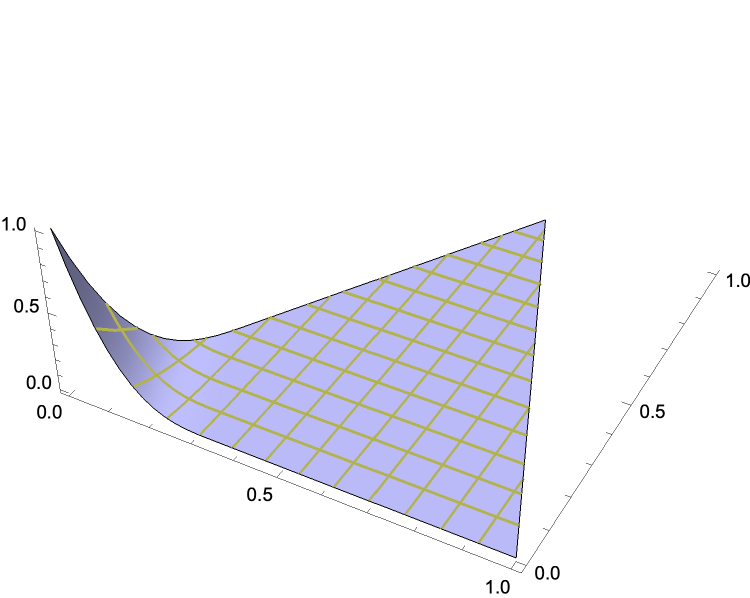}
\includegraphics[width=3.8cm]{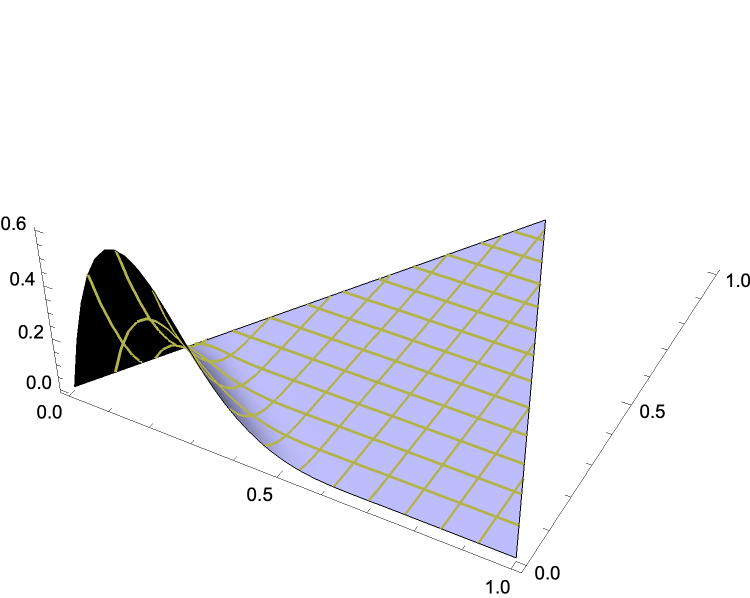}
\includegraphics[width=3.8cm]{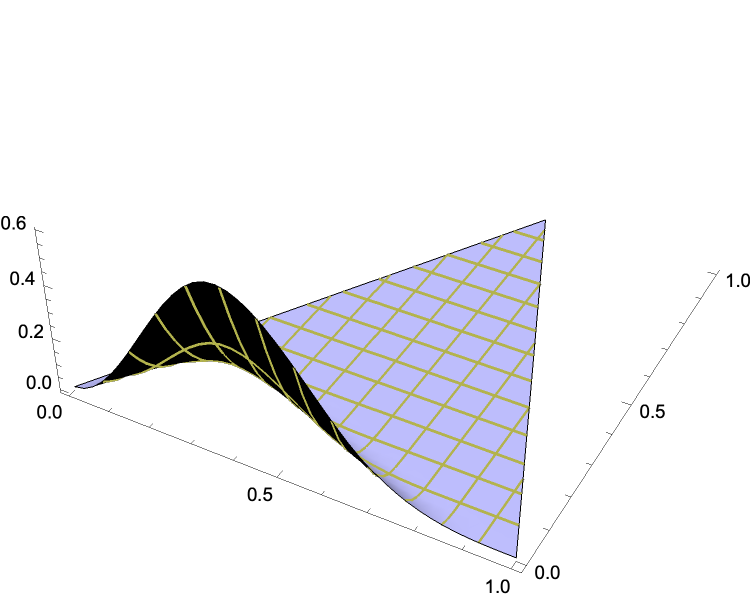}
\caption{The simplex spline basis functions $B_1$, $B_4$, and $B_{10}$.}
\label{fig:B1-B4-B10}
\medskip
\centering
\includegraphics[width=3.8cm]{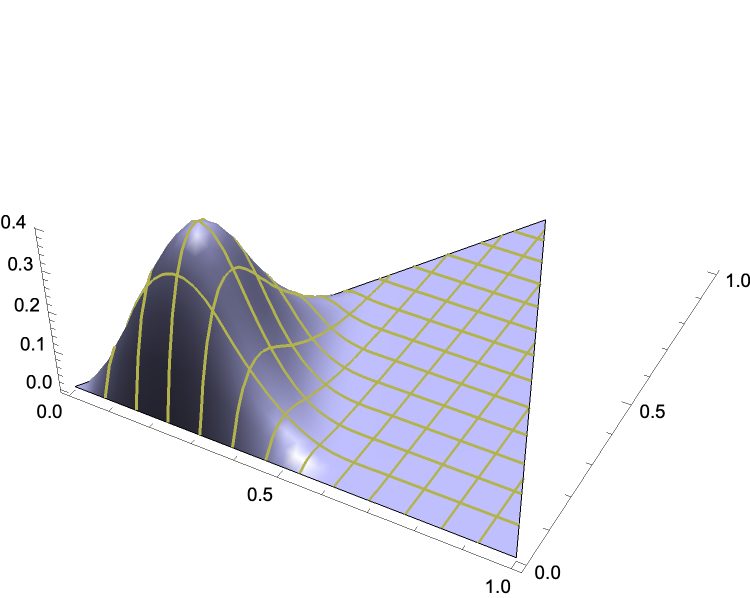}
\includegraphics[width=3.8cm]{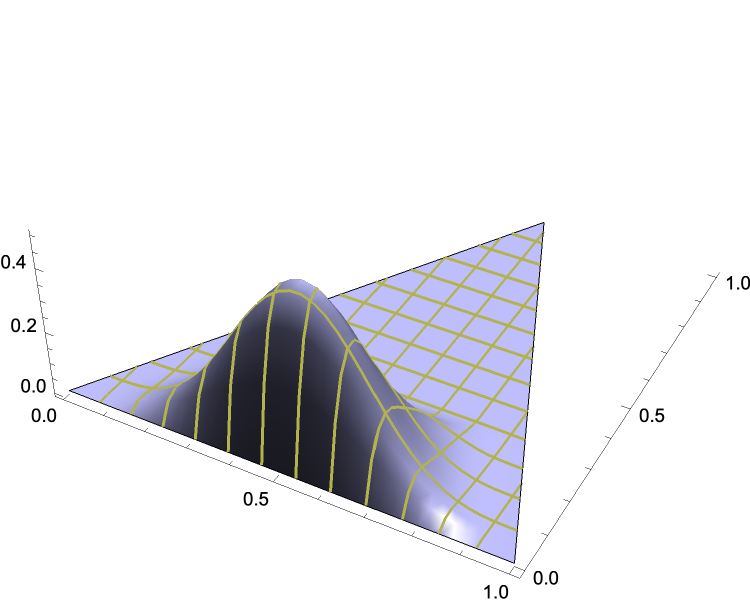}
\includegraphics[width=3.8cm]{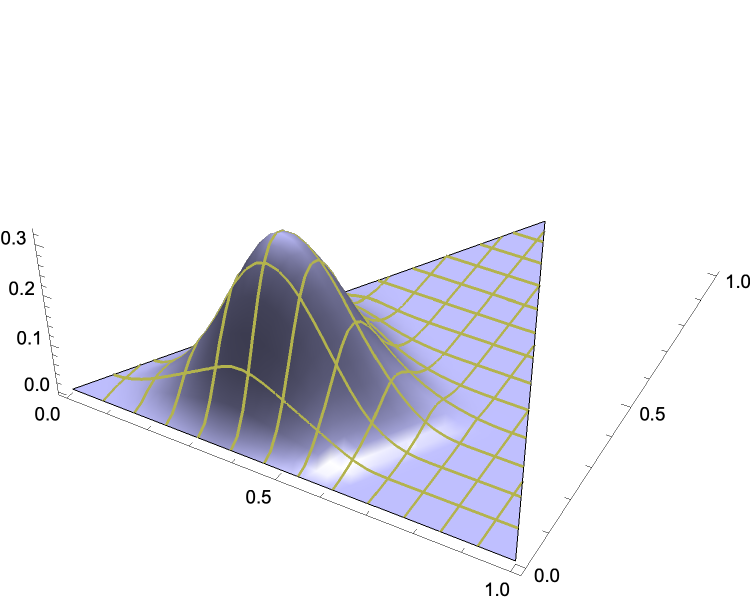}
\caption{The simplex spline basis functions $B_{16}$, $B_{19}$, and $B_{22}$.}
\label{fig:B16-B19-B22}
\medskip
\centering
\includegraphics[width=3.8cm]{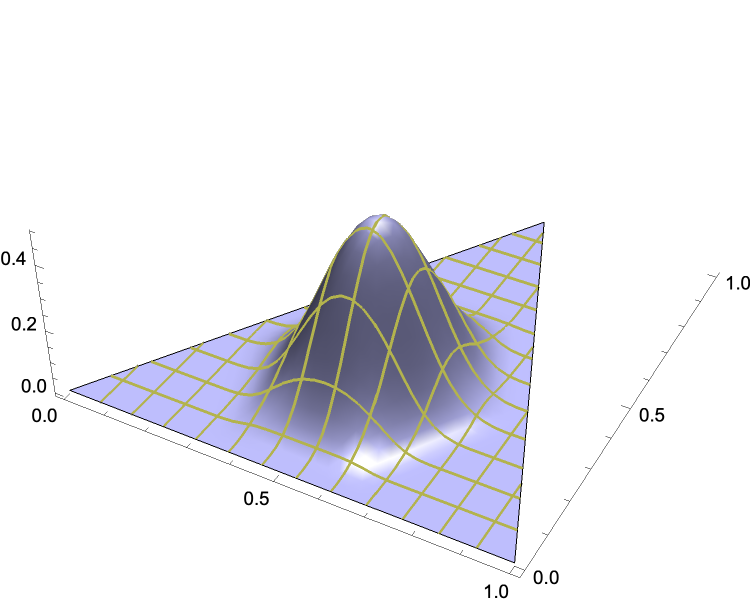}
\caption{The simplex spline basis function $B_{28}$.}
\label{fig:B28}
\end{figure}

For the exposition in the next sections, it is useful to analyze the structure of the collocation matrix, denoted by $\matC$, of the basis \eqref{eq:basisB} for the Hermite interpolation problem stated in Proposition~\ref{pro:hermite}. Let us consider the following ordering of the Hermite interpolation conditions by defining the linear functionals $\lambda_i$, $ i=1,\ldots, 28$. The first functionals
$\lambda_1,\ldots,\lambda_{18}$ are related to the vertices,
\begin{equation} \label{eq:lambda-1}
\begin{alignedat}{4}
\lambda_1(f)&:=f(\vp_1), &\quad \lambda_2(f)&:=f(\vp_2), &\quad \lambda_3(f)&:=f(\vp_3),\\
\lambda_4(f)&:=D_{x}f(\vp_1), & \quad
\lambda_5(f)&:=D_{y}f(\vp_1), & \quad
\lambda_6(f)&:=D_{x}f(\vp_2), \\
\lambda_7(f)&:=D_{y}f(\vp_2), &
\lambda_8(f)&:=D_{x}f(\vp_3), &
\lambda_9(f)&:=D_{y}f(\vp_3), \\
\lambda_{10}(f)&:=D_{x}^2f(\vp_1), & \quad 
\lambda_{11}(f)&:=D_{y}^2f(\vp_1), & \quad 
\lambda_{12}(f)&:=D_{x}^2f(\vp_2), \\ 
\lambda_{13}(f)&:=D_{y}^2f(\vp_2), & 
\lambda_{14}(f)&:=D_{x}^2f(\vp_3), & 
\lambda_{15}(f)&:=D_{y}^2f(\vp_3), \\ 
\lambda_{16}(f)&:=D_{x}D_{y}f(\vp_1), &
\lambda_{17}(f)&:=D_{x}D_{y}f(\vp_2), &
\lambda_{18}(f)&:=D_{x}D_{y}f(\vp_3);
\end{alignedat}
\end{equation}
$\lambda_{19},\ldots,\lambda_{27}$ are related to the edges,
\begin{equation} \label{eq:lambda-2}
\begin{alignedat}{4}
\lambda_{19}(f)&:=D_{\vn_3}f(\vq_3), &
\lambda_{20}(f)&:=D_{\vn_1}f(\vq_1),&
\lambda_{21}(f)&:=D_{\vn_2}f(\vq_2), \\
\lambda_{22}(f)&:=D_{\vn_3}^2f(\vp_{3,1}), & 
\lambda_{23}(f)&:=D_{\vn_2}^2f(\vp_{2,1}), & 
\lambda_{24}(f)&:=D_{\vn_1}^2f(\vp_{1,2}), \\ 
\lambda_{25}(f)&:=D_{\vn_3}^2f(\vp_{3,2}), & \quad 
\lambda_{26}(f)&:=D_{\vn_2}^2f(\vp_{2,3}), & \quad 
\lambda_{27}(f)&:=D_{\vn_1}^2f(\vp_{1,3}); 
\end{alignedat}
\end{equation}
and the final $\lambda_{28}$ is related to the triangle,
\begin{equation}\label{eq:lambda-3}
\lambda_{28}(f) :=f(\vq).
\end{equation}
Then, the collocation matrix $\matC$ is formed by
$$
c_{i,j}:=\lambda_i(B_j), \quad i,j=1,\dots, 28.
$$
From the support and the smoothness properties of simplex splines, taking into account the knot configurations in Figure~\ref{fig:allknots}, we deduce the sparsity structure of $\matC$ depicted in Table~\ref{tab:hermiteB}, where the symbols {\entry} stand for possible nonzero entries. The exact values of the entries of $\matC$ can be obtained via standard elementary calculus by considering the values in \cite[Tables~3--5]{LycheMS22}; see also Example~\ref{example}. Proposition~\ref{pro:hermite} implies that $\matC$ is not singular.

\begin{remark}
 The collocation matrix presented in \cite[Tables~3--5]{LycheMS22} is defined in terms of linear functionals $\rho_i$, $i=1,\ldots,28$. These functionals are strongly related to the $\lambda_i$, $i=1,\ldots,28$ in \eqref{eq:lambda-1}--\eqref{eq:lambda-3} but their derivatives are taken in triangle-dependent directions instead of axis-aligned directions or edge-normal directions.
\end{remark}

\begin{examples}\label{example}
Let us consider the triangle with vertices $(0,0)$, $(h,0)$, and $(0,h)$. The entries of the corresponding collocation matrix $\matC$ are expressed in Table~\ref{tab:hermiteB2-reference-tri-h}.
\end{examples}

\subsection{Reduced spaces}

With the aim of achieving a parsimonious approach to construct $C^2$ cubic spline functions on a given triangulation $\cT$, we are interested in defining subspaces of $\spS_3^2(\cT_{\WS})$
\begin{itemize}
 \item that contain $\spP_3$, which is a necessary condition for attaining full approximation order;
 \item whose elements can be locally identified by assigning Hermite degrees of freedom only to the vertices and edges of $\cT$, or even better, only to the vertices of $\cT$;
 \item that allow for a simple local spline representation on each $\Delta$ of $\cT$, to avoid the need of working with separate polynomial pieces on each polygonal subregion of $\cT_{\WS}$.
\end{itemize}
We first develop a framework for constructing such subspaces of $\spS_3^2(\Delta_\WS)$ in Section~\ref{sec:reduced-spaces-local} and then we extend the framework to $\spS_3^2(\cT_{\WS})$ in Section~\ref{sec:reduced-spaces-global}.

\begin{landscape}
\begin{table}[t!]
\renewcommand{\arraystretch}{1.3}
\centering{\scriptsize
\setlength{\tabcolsep}{3.5pt}
\begin{tabular}{|c|ccc|cccccc|ccccccccc|ccc|cccccc|c|}
\hline
& $B_1$ & $B_2$ & $B_3$ & $B_4$ & $B_5$ & $B_6$ & $B_7$ & $B_8$ & $B_9$ & $B_{10}$
& $B_{11}$ & $B_{12}$ & $B_{13}$ & $B_{14}$ & $B_{15}$ & $B_{16}$ & $B_{17}$ & $B_{18}$ & $B_{19}$ &$B_{20}$ &$B_{21}$ & $B_{22}$ &$B_{23}$ &$B_{24}$ & $B_{25}$ &$B_{26}$ &$B_{27}$ & $B_{28}$ 
\\
\hline
$\lambda_{1}$ & \entry & 0 & 0 & 0 & 0 & 0 & 0 & 0 & 0 & 0 & 0 & 0 & 0 & 0 & 0 & 0 & 0 & 0 & 0 & 0 & 0 & 0 & 0 & 0 & 0 & 0 & 0 & 0 \\
$\lambda_{2}$ & 0 & \entry & 0 & 0 & 0 & 0 & 0 & 0 & 0 & 0 & 0 & 0 & 0 & 0 & 0 & 0 & 0 & 0 & 0 & 0 & 0 & 0 & 0 & 0 & 0 & 0 & 0 & 0 \\
$\lambda_{3}$ & 0 & 0 & \entry & 0 & 0 & 0 & 0 & 0 & 0 & 0 & 0 & 0 & 0 & 0 & 0 & 0 & 0 & 0 & 0 & 0 & 0 & 0 & 0 & 0 & 0 & 0 & 0 & 0 \\
\hline
$\lambda_{4}$ & \entry & 0 & 0 & \entry & \entry & 0 & 0 & 0 & 0 & 0 & 0 & 0 & 0 & 0 & 0 & 0 & 0 & 0 & 0 & 0 & 0 & 0 & 0 & 0 & 0 & 0 & 0 & 0 \\
$\lambda_{5}$ & \entry & 0 & 0 & \entry & \entry & 0 & 0 & 0 & 0 & 0 & 0 & 0 & 0 & 0 & 0 & 0 & 0 & 0 & 0 & 0 & 0 & 0 & 0 & 0 & 0 & 0 & 0 & 0 \\
$\lambda_{6}$ & 0 & \entry & 0 & 0 & 0 & \entry & \entry & 0 & 0 & 0 & 0 & 0 & 0 & 0 & 0 & 0 & 0 & 0 & 0 & 0 & 0 & 0 & 0 & 0 & 0 & 0 & 0 & 0 \\
$\lambda_{7}$ & 0 & \entry & 0 & 0 & 0 & \entry & \entry & 0 & 0 & 0 & 0 & 0 & 0 & 0 & 0 & 0 & 0 & 0 & 0 & 0 & 0 & 0 & 0 & 0 & 0 & 0 & 0 & 0 \\ 
$\lambda_{8}$ & 0 & 0 & \entry & 0 & 0 & 0 & 0 & \entry & \entry & 0 & 0 & 0 & 0 & 0 & 0 & 0 & 0 & 0 & 0 & 0 & 0 & 0 & 0 & 0 & 0 & 0 & 0 & 0 \\
$\lambda_{9}$ & 0 & 0 & \entry & 0 & 0 & 0 & 0 & \entry & \entry & 0 & 0 & 0 & 0 & 0 & 0 & 0 & 0 & 0 & 0 & 0 & 0 & 0 & 0 & 0 & 0 & 0 & 0 & 0 \\
\hline
$\lambda_{10}$& \entry & 0 & 0 & \entry & \entry & 0 & 0 & 0 & 0 & \entry & \entry & 0 & 0 & 0 & 0 & \entry & 0 & 0  & 0 & 0 & 0 & 0 & 0 & 0 & 0 & 0 & 0 & 0\\
$\lambda_{11}$& \entry & 0 & 0 & \entry & \entry & 0 & 0 & 0 & 0 & \entry & \entry & 0 & 0 & 0 & 0 & \entry & 0 & 0 & 0 & 0 & 0 & 0 & 0 & 0 & 0 & 0 & 0 & 0  \\
$\lambda_{12}$& 0 & \entry & 0 & 0 & 0 & \entry & \entry & 0 & 0 & 0 & 0 & \entry & \entry & 0 & 0 & 0 & \entry & 0 & 0 & 0 & 0 & 0 & 0 & 0 & 0 & 0 & 0 & 0  \\
$\lambda_{13}$& 0 & \entry & 0 & 0 & 0 & \entry & \entry & 0 & 0 & 0 & 0 & \entry & \entry & 0 & 0 & 0 & \entry & 0 & 0 & 0 & 0 & 0 & 0 & 0 & 0 & 0 & 0 & 0 \\
$\lambda_{14}$& 0 & 0 & \entry & 0 & 0 & 0 & 0 & \entry & \entry & 0 & 0 & 0 & 0 & \entry & \entry & 0 & 0 & \entry & 0 & 0 & 0 & 0 & 0 & 0 & 0 & 0 & 0 & 0 \\
$\lambda_{15}$& 0 & 0 & \entry & 0 & 0 & 0 & 0 & \entry & \entry & 0 & 0 & 0 & 0 & \entry & \entry & 0 & 0 & \entry & 0 & 0 & 0 & 0 & 0 & 0 & 0 & 0 & 0 & 0  \\
$\lambda_{16}$& \entry & 0 & 0 & \entry & \entry & 0 & 0 & 0 & 0 & \entry & \entry & 0 & 0 & 0 & 0 & \entry & 0 & 0 & 0 & 0 & 0 & 0 & 0 & 0 & 0 & 0 & 0 & 0 \\
$\lambda_{17}$& 0 & \entry & 0 & 0 & 0 & \entry & \entry & 0 & 0 & 0 & 0 & \entry & \entry & 0 & 0 & 0 & \entry & 0 & 0 & 0 & 0 & 0 & 0 & 0 & 0 & 0 & 0 & 0 \\
$\lambda_{18}$& 0 & 0 & \entry & 0 & 0 & 0 & 0 & \entry & \entry & 0 & 0 & 0 & 0 & \entry & \entry & 0 & 0 & \entry & 0 & 0 & 0 & 0 & 0 & 0 & 0 & 0 & 0 & 0\\
\hline
$\lambda_{19}$& 0 & 0 & 0 & \entry & 0 & 0 & \entry & 0 & 0 & \entry & 0 & 0 & \entry & 0 & 0 & \entry & \entry & 0 & \entry & 0 & 0 & 0 & 0 & 0 & 0 & 0 & 0 & 0  \\
$\lambda_{20}$& 0 & 0 & 0 & 0 & 0 & \entry & 0 & 0 & \entry & 0 & 0 & \entry & 0 & 0 & \entry & 0 & \entry & \entry & 0 & \entry & 0 & 0 & 0 & 0 & 0 & 0 & 0 & 0\\
$\lambda_{21}$& 0 & 0 & 0 & 0 & \entry & 0 & 0 & \entry & 0 & 0 & \entry & 0 & 0 & \entry & 0 & \entry & 0 & \entry & 0 & 0 & \entry & 0 & 0 & 0 & 0 & 0 & 0 & 0\\
\hline
$\lambda_{22}$& 0 & 0 & 0 & \entry & 0 & 0 & 0 & 0 & 0 & \entry & 0 & 0 & \entry & 0 & 0 & \entry & 0 & 0 & \entry & 0 & 0 & \entry & \entry & 0 & 0 & 0 & 0 & 0\\
$\lambda_{23}$& 0 & 0 & 0 & 0 & \entry & 0 & 0 & 0 & 0 & 0 & \entry & 0 & 0 & \entry & 0 & \entry & 0 & 0 & 0 & 0 & \entry & \entry & \entry & 0 & 0 & 0 & 0 & 0\\
$\lambda_{24}$& 0 & 0 & 0 & 0 & 0 & \entry & 0 & 0 & 0 & 0 & 0 & \entry & 0 & 0 & \entry & 0 & \entry & 0 & 0 & \entry & 0 & 0 & 0 & \entry & \entry & 0 & 0 & 0\\
$\lambda_{25}$& 0 & 0 & 0 & 0 & 0 & 0 & \entry & 0 & 0 & \entry & 0 & 0 & \entry & 0 & 0 & 0 & \entry & 0 & \entry & 0 & 0 & 0 & 0 & \entry & \entry & 0 & 0 & 0\\
$\lambda_{26}$& 0 & 0 & 0 & 0 & 0 & 0 & 0 & \entry & 0 & 0 & \entry & 0 & 0 & \entry & 0 & 0 & 0 & \entry & 0 & 0 & \entry & 0 & 0 & 0 & 0 & \entry & \entry & 0 \\
$\lambda_{27}$& 0 & 0 & 0 & 0 & 0 & 0 & 0 & 0 & \entry & 0 & 0 & \entry & 0 & 0 & \entry & 0 & 0 & \entry & 0 & \entry & 0 & 0 & 0 & 0 & 0 & \entry & \entry & 0\\
\hline
$\lambda_{28}$& 0 & 0 & 0 & 0 & 0 & 0 & 0 & 0 & 0 & 0 & 0 & 0 & 0 & 0 & 0 & 0 & 0 & 0 & 0 & 0 & 0 & \entry & \entry & \entry & \entry & \entry & \entry & \entry\\
\hline
\end{tabular}}
\caption{General sparsity pattern of $\lambda_i(B_j)$ for $i,j=1,\ldots,28$, where $\lambda_i$ is defined in \eqref{eq:lambda-1}, \eqref{eq:lambda-2}, and \eqref{eq:lambda-3}. }
\label{tab:hermiteB}
\end{table}
\end{landscape}
\begin{landscape}
\begin{table}[t!]
\renewcommand{\arraystretch}{1.25}
\centering{\scriptsize
\setlength{\tabcolsep}{2.5pt}
\begin{tabular}{|c|ccc|cccccc|ccccccccc|ccc|cccccc|c|}
\hline
&$B_1$&$B_2$&$B_3$&$B_4$&$B_5$&$B_6$&$B_7$&$B_8$&$B_9$&$B_{10}$&$B_{11}$&$B_{12}$&$B_{13}$&$B_{14}$&$B_{15}$&$B_{16}$&$B_{17}$&$B_{18}$&$B_{19}$&$B_{20}$&$B_{21}$&$B_{22}$&$B_{23}$&$B_{24}$&$B_{25}$&$B_{26}$&$B_{27}$&$B_{28}$\\
\hline	
$\lambda_1$& 1 & 0 & 0 & 0 & 0 & 0 & 0 & 0 & 0 & 0 & 0 & 0 & 0 & 0 & 0 & 0 & 0 & 0 & 0 & 0 & 0 & 0 & 0 & 0 & 0 & 0 & 0 & 0 \\
$\lambda_2$ &0 & 1 & 0 & 0 & 0 & 0 & 0 & 0 & 0 & 0 & 0 & 0 & 0 & 0 & 0 & 0 & 0 & 0 & 0 & 0 & 0 & 0 & 0 & 0 & 0 & 0 & 0 & 0 \\
$\lambda_3$& 0 & 0 & 1 & 0 & 0 & 0 & 0 & 0 & 0 & 0 & 0 & 0 & 0 & 0 & 0 & 0 & 0 & 0 \
& 0 & 0 & 0 & 0 & 0 & 0 & 0 & 0 & 0 & 0 \\
\hline
$\lambda_4$&-$\frac{9}{h}$ & 0 & 0 & $\frac{9}{h}$ & 0 & 0 & 0 & 0 & 0 & 0 & 0 & 0 & 0 & 0& 0 & 0 & 0 & 0 & 0 & 0 & 0 & 0 & 0 & 0 & 0 & 0 & 0 & 0 \\
$\lambda_5$&-$\frac{9}{h}$ & 0 & 0 & 0 & $\frac{9}{h}$ & 0 & 0 & 0 & 0 & 0 & 0 & 0 & 0 & 0  & 0 & 0 & 0 & 0 & 0 & 0 & 0 & 0 & 0 & 0 & 0 & 0 & 0 & 0 \\
$\lambda_6$&0 & $\frac{9}{h}$ & 0 & 0 & 0 & 0 & -$\frac{9}{h}$ & 0 & 0 & 0 & 0 & 0 & 0 & 0 & 0 & 0 & 0 & 0 & 0 & 0 & 0 & 0 & 0 & 0 & 0 & 0 & 0 & 0 \\
$\lambda_7$& 0 & 0 & 0 & 0 & 0 & $\frac{9}{h}$ & -$\frac{9}{h}$ & 0 & 0 & 0 & 0 & 0 & 0 & 0 & 0 & 0 & 0 & 0 & 0 & 0 & 0 & 0 & 0 & 0 & 0 & 0 & 0 & 0 \\
$\lambda_8$& 0 & 0 & 0 & 0 & 0 & 0 & 0 & -$\frac{9}{h}$ & $\frac{9}{h}$ & 0 & 0 & 0 & 0 & 0& 0 & 0 & 0 & 0 & 0 & 0 & 0 & 0 & 0 & 0 & 0 & 0 & 0 & 0 \\
$\lambda_9$&0 & 0 & $\frac{9}{h}$ & 0 & 0 & 0 & 0 & -$\frac{9}{h}$ & 0 & 0 & 0 & 0 & 0 & 0& 0 & 0 & 0 & 0 & 0 & 0 & 0 & 0 & 0 & 0 & 0 & 0 & 0 & 0 \\
\hline
$\lambda_{10}$& $\frac{54}{h^2}$ & 0 & 0 & -$\frac{81}{h^2}$ & 0 & 0 & 0 & 0 & 0 & $\frac{27}{h^2}$ & 0 & 0 & 0 & 0 & 0 & 0 & 0 & 0 & 0 & 0 & 0 & 0 & 0 & 0 & 0 & 0 & 0 & 0 \\
$\lambda_{11}$&$\frac{54}{h^2}$ & 0 & 0 & 0 & -$\frac{81}{h^2}$ & 0 & 0 & 0 & 0 & 0 & $\frac{27}{h^2}$ & 0 & 0 & 0 & 0 & 0 & 0 & 0 & 0 & 0 & 0 & 0 & 0 & 0 & 0 & 0 & 0 & 0 \\
$\lambda_{12}$& 0 & $\frac{54}{h^2}$ & 0 & 0 & 0 & 0 & -$\frac{81}{h^2}$ & 0 & 0 & 0 & 0 & 0 &$\frac{27}{h^2}$ & 0 & 0 & 0 & 0 & 0 & 0 & 0 & 0 & 0 & 0 & 0 & 0 & 0 & 0 &0 \\
$\lambda_{13}$& 0 & 0 & 0 & 0 & 0 & $\frac{27}{h^2}$ & $\frac{27}{h^2}$ & 0 & 0 & 0 & 0 & $\frac{27}{h^2}$ & $\frac{27}{h^2}$ & 0 & 0 & 0 &
-$\frac{108}{h^2}$ & 0 & 0 & 0 & 0 & 0 & 0 & 0 & 0 & 0 & 0 & 0 \\
$\lambda_{14}$& 0 & 0 & 0 & 0 & 0 & 0 & 0 & $\frac{27}{h^2}$ & $\frac{27}{h^2}$ & 0 & 0 & 0 & 0 & $\frac{27}{h^2}$ & $\frac{27}{h^2}$ & 0 & 0 & -$\frac{108}{h^2}$ & 0 & 0 & 0
& 0 & 0 & 0 & 0 & 0 & 0 & 0 \\
$\lambda_{15}$& 0 & 0 & $\frac{54}{h^2}$ & 0 & 0 & 0 & 0 & -$\frac{81}{h^2}$ & 0 & 0 & 0 & 0 & 0 & $\frac{27}{h^2}$ & 0 & 0 & 0 & 0 & 0 & 0 & 0 & 0 & 0 & 0 & 0 & 0 & 0 & 0 \\
$\lambda_{16}$&$\frac{54}{h^2}$ & 0 & 0 & -$\frac{54}{h^2}$ & -$\frac{54}{h^2}$ & 0 & 0 & 0 & 0& 0 & 0 & 0 & 0 & 0 & 0 & $\frac{54}{h^2}$ & 0 & 0 & 0 & 0 & 0 & 0 & 0 & 0& 0 & 0 & 0 & 0 \\
$\lambda_{17}$&0 & 0 & 0 & 0 & 0 & $\frac{54}{h^2}$ & -$\frac{27}{h^2}$ & 0 & 0 & 0 & 0 & 0 & $\frac{27}{h^2}$ & 0 & 0 & 0 & -$\frac{54}{h^2}$ & 0 & 0 & 0 & 0 & 0 & 0 & 0 & 0 & 0 & 0 & 0 \\
$\lambda_{18}$&0 & 0 & 0 & 0 & 0 & 0 & 0 & -$\frac{27}{h^2}$ & $\frac{54}{h^2}$ & 0 & 0 & 0 & 0 & $\frac{27}{h^2}$ & 0 & 0 & 0 & -$\frac{54}{h^2}$ & 0 & 0 & 0 & 0 & 0 & 0& 0 & 0 & 0 & 0 \\
\hline
$\lambda_{19}$& 0 & 0 & 0 & $\frac{9}{8 h}$ & 0 & 0 & $\frac{9}{16 h}$ & 0 & 0 & $\frac{27}{4 h}$ & 0 & 0 & $\frac{81}{16 h}$ & 0 & 0 & -$\frac{9}{8 h}$ & -$\frac{9}{8 h}$ & 0 &
-$\frac{45}{4 h}$ & 0 & 0 & 0 & 0 & 0 & 0 & 0 & 0 & 0 \\
$\lambda_{20}$&0 & 0 & 0 & 0 & 0 & $\frac{27}{16 \sqrt{2} h}$ & 0 & 0 & $\frac{27}{16\sqrt{2} h}$ & 0 & 0 & $\frac{189}{16 \sqrt{2} h}$ & 0 & 0 & $\frac{189}{16
    \sqrt{2} h}$ & 0 & -$\frac{9}{4 \sqrt{2} h}$ & -$\frac{9}{4 \sqrt{2} h}$ & 0 &
-$\frac{45}{2 \sqrt{2} h}$ & 0 & 0 & 0 & 0 & 0 & 0 & 0 & 0 \\
$\lambda_{21}$&0 & 0 & 0 & 0 & $\frac{9}{8 h}$ & 0 & 0 & $\frac{9}{16 h}$ & 0 & 0 &$\frac{27}{4 h}$ & 0 & 0 & $\frac{81}{16 h}$ & 0 & -$\frac{9}{8 h}$ & 0 &
-$\frac{9}{8 h}$ & 0 & 0 & -$\frac{45}{4 h}$ & 0 & 0 & 0 & 0 & 0 & 0 & 0 \\
\hline
$\lambda_{22}$& 0 & 0 & 0 & $\frac{54}{h^2}$ & 0 & 0 & 0 & 0 & 0 & 0 & 0 & 0 & $\frac{36}{h^2}$& 0 & 0 & -$\frac{81}{h^2}$ & 0 & 0 & -$\frac{90}{h^2}$ & 0 & 0 &
$\frac{54}{h^2}$ & $\frac{27}{h^2}$ & 0 & 0 & 0 & 0 & 0 \\
$\lambda_{23}$&0 & 0 & 0 & 0 & $\frac{54}{h^2}$ & 0 & 0 & 0 & 0 & 0 & 0 & 0 & 0 &$\frac{36}{h^2}$ & 0 & -$\frac{81}{h^2}$ & 0 & 0 & 0 & 0 & -$\frac{90}{h^2}$ &
$\frac{27}{h^2}$ & $\frac{54}{h^2}$ & 0 & 0 & 0 & 0 & 0 \\
$\lambda_{24}$&0 & 0 & 0 & 0 & 0 & $\frac{243}{4 h^2}$ & 0 & 0 & 0 & 0 & 0 & $\frac{171}{4 h^2}$ & 0 & 0 & $\frac{225}{2 h^2}$ & 0 & -$\frac{108}{h^2}$ & 0 & 0 &
-$\frac{270}{h^2}$ & 0 & 0 & 0 & $\frac{108}{h^2}$ & $\frac{54}{h^2}$ & 0 & 0 & 0 \\
$\lambda_{25}$& 0 & 0 & 0 & 0 & 0 & 0 & $\frac{27}{2 h^2}$ & 0 & 0 & $\frac{81}{h^2}$ & 0 & 0 & $\frac{63}{2 h^2}$ & 0 & 0 & 0 & -$\frac{27}{h^2}$ & 0 & -$\frac{180}{h^2}$ & 0
& 0 & 0 & 0 & $\frac{27}{h^2}$ & $\frac{54}{h^2}$ & 0 & 0 & 0 \\
$\lambda_{26}$& 0 & 0 & 0 & 0 & 0 & 0 & 0 & $\frac{27}{2 h^2}$ & 0 & 0 & $\frac{81}{h^2}$ & 0 & 0 & $\frac{63}{2 h^2}$ & 0 & 0 & 0 & -$\frac{27}{h^2}$ & 0 & 0 &
-$\frac{180}{h^2}$ & 0 & 0 & 0 & 0 & $\frac{54}{h^2}$ & $\frac{27}{h^2}$ & 0 \\
$\lambda_{27}$& 0 & 0 & 0 & 0 & 0 & 0 & 0 & 0 & $\frac{243}{2 h^2}$ & 0 & 0 & $\frac{225}{h^2}$ & 0 & 0 & $\frac{171}{2 h^2}$ & 0 & 0 & -$\frac{216}{h^2}$ & 0 &
-$\frac{540}{h^2}$ & 0 & 0 & 0 & 0 & 0 & $\frac{108}{h^2}$ & $\frac{216}{h^2}$
& 0 \\
\hline
$\lambda_{28}$& 0 & 0 & 0 & 0 & 0 & 0 & 0 & 0 & 0 & 0 & 0 & 0 & 0 & 0 & 0 & 0 & 0 & 0 & 0 & 0 & 0 & $\frac{1}{12}$ & $\frac{1}{12}$ & $\frac{1}{12}$ & $\frac{1}{12}$ &
$\frac{1}{12}$ & $\frac{1}{12}$ & $\frac{1}{2}$ \\
\hline
\end{tabular}
\caption{Values of $\lambda_i(B_j)$  for $i,j=1,\dots,28$,  where $\lambda_i$ is defined in \eqref{eq:lambda-1}, \eqref{eq:lambda-2}, and \eqref{eq:lambda-3}, for the triangle $\Delta:=\langle (0,0), (h,0), (0,h)\rangle$.}
\label{tab:hermiteB2-reference-tri-h}}
\end{table}
\end{landscape}

\section{Reduced spaces of $\spS_3^2(\Delta_\WS)$}
\label{sec:reduced-spaces-local}
As known in the finite element literature, to remove Hermite degrees of freedom associated with interior points of a macro-triangle $\Delta$ or with its edges, we may express them as a linear combination, with suitable coefficients, of the degrees of freedom associated with the vertices. For determining subspaces that possess full approximation order it is necessary that the considered reduced expressions are exact for cubic polynomials. 

\subsection{Peeling away Hermite degrees of freedom}
\label{sec:peeling}
We start by discussing briefly some reasonable strategies to peel away degrees of freedom from the Hermite data in \eqref{eq:lambda-1}--\eqref{eq:lambda-3} that can be adopted for any macro-triangle $\Delta$. 
Let ${\vn_3}$ and ${\vt_3}$ be the normal and tangent direction of the edge $\vp_1\vp_2$. We remark that from the available Hermite data at $\vp_1$ and $\vp_2$ ($\lambda_j$, $j=4,\ldots,7$, $j=10,\ldots,13$, and $j=16,17$) we can directly compute the values of the normal and tangent derivative at $\vp_1$ and $\vp_2$, and also the values of the normal and tangent derivative of the normal derivative at the same points, namely
\begin{equation}
\label{eq:data-normal}
D_{\vn_3}s(\vp_1), \quad D_{\vn_3}s(\vp_2), \quad D_{\vt_3} D_{\vn_3}s(\vp_1), \quad D_{\vt_3} D_{\vn_3}s(\vp_2),
\end{equation}
and
\begin{equation}
\label{eq:data-normal2}
D^2_{\vn_3}s(\vp_1), \quad D^2_{\vn_3}s(\vp_2).
\end{equation}
Therefore, we consider the following reduction strategies that preserve cubic polynomial reproduction.
\begin{itemize}
\item {\bf Second normal derivatives associated with the edges} ($\lambda_{22}$--$\lambda_{27}$).\\
Along the edge $\vp_1\vp_2$ the second derivative in the normal direction is a $C^0$ piecewise linear function  with interior break points at $\vp_{3,1}$ and $\vp_{3.2}$. The most natural way to remove the degrees of freedom given by $\lambda_{22}$ and $\lambda_{25}$ is imposing that the second normal derivative along the edge $\vp_1\vp_2$ is the straight line interpolating the data \eqref{eq:data-normal2}. 
The other two edges of $\Delta$ can be treated similarly. 
\item {\bf First normal derivatives associated with the edges} ($\lambda_{19}$--$\lambda_{21}$).\\
Along the edge $\vp_1\vp_2$ the first derivative in the normal direction is a $C^1$ piecewise quadratic function with interior break points at $\vp_{3,1}$ and $\vp_{3.2}$. 
We can remove the degree of freedom given by $\lambda_{19}$ via one of the following approaches:
\begin{itemize} 
\item imposing that  the value of the normal derivative at $\vq_3$ is obtained by averaging the values at $\vq_3$ of the two (univariate) quadratic polynomials interpolating
$$
D_{\vn_3}s(\vp_1), \quad  D_{\vn_3}s(\vp_2), \quad D_{\vt_3} D_{\vn_3}s(\vp_1),
$$
and
$$
D_{\vn_3}s(\vp_1), \quad D_{\vn_3}s(\vp_2), \quad  D_{\vt_3} D_{\vn_3}s(\vp_2),
$$
respectively; 
\item imposing that the value of the normal derivative at $\vq_3$ is obtained by averaging the values at $\vq_3$ of the $C^1$ piecewise quadratic function with (internal) break point at $\vp_{3,1}$ which interpolates the data \eqref{eq:data-normal} and of the $C^1$ piecewise quadratic function with (internal) break point at $\vp_{3,2}$ which interpolates again the data \eqref{eq:data-normal}. 
\end{itemize}
The other two edges of $\Delta$ can be treated similarly.  
\item {\bf Value at the interior point} ($\lambda_{28}$).\\
The selection of a suitable strategy for computing the interior value at $\vq$ as a function of the Hermite data at the vertices (and possibly also along the edges) is more difficult because there are too many ways that could be used to fix it and still preserve cubic polynomial reproduction. In Section~\ref{sec:interior-dof} we will describe a general strategy to remove the interior Hermite degree of freedom based on the simplex spline representation of the space $\spS_3^2(\Delta_\WS)$; this strategy is able to preserve the symmetry of the (simplex spline) basis functions.   
\end{itemize}

\subsection{Hermite and simplex spline representation of subspaces of $\spS_3^2(\Delta_{\WS})$}
\label{sec:hermite-simplex}
We discuss how to represent subspaces of $\spS_3^2(\Delta_{\WS})$ that are obtained by peeling away certain Hermite degrees of freedom as discussed in Section~\ref{sec:peeling}, by means of the simplex spline basis in \eqref{eq:basisB}. The use of the simplex spline basis is beneficial because it intrinsically takes care of the complex geometry of the split and can be constructed without a piecewise polynomial treatment on every polygonal subregion of $\Delta_\WS$.

Let us consider the local Hermite basis of $\spS_3^2(\Delta_{\WS})$ associated with the Hermite interpolation problem stated in Proposition~\ref{pro:hermite}:
\begin{equation*}
\basisH := [H_1,\ldots,H_{28}],
\end{equation*}
i.e., the vector of functions in $\spS_3^2(\Delta_{\WS})$ such that 
$$
\lambda_i(H_j)=\delta_{i,j}, \quad i,j=1,\dots,28,
$$
where $\delta_{i,j}$ denotes the Kronecker delta and the linear functionals $\lambda_i$ are defined in \eqref{eq:lambda-1}--\eqref{eq:lambda-3}.
Then, observing that both $\basisB$ and $\basisH$ are row vectors, we have  
\begin{equation}
\label{eq:HB}
\basisB=\basisH \cdot \matC,
\end{equation}
where $\matC$ represents the collocation matrix of the basis \eqref{eq:basisB} for the Hermite interpolation problem stated in Proposition~\ref{pro:hermite}, whose structure is depicted in Table~\ref{tab:hermiteB}.

For the following we need to fix some notation. For $m\in\{1,\dots, 28\}$ we write
$$\matXmm{m} := \begin{bmatrix} 
\matXmm{m}_{11} \\[0.1cm] \matXmm{m}_{21} 
\end{bmatrix}\in\RR^{28 \times m}, \quad
\matXme{m} := \begin{bmatrix} 
\matXmm{m}_{11} & \matXmm{m}_{12} \\[0.1cm] 
\matXmm{m}_{21} & \matXmm{m}_{22}
\end{bmatrix}\in\RR^{28 \times 28}, $$
where 
$$
\matXmm{m}_{11}\in\RR^{m \times m}, \quad \matXmm{m}_{12}\in\RR^{m \times \mbar}, \quad
\matXmm{m}_{21}\in\RR^{\mbar \times m}, \quad \matXmm{m}_{22}\in\RR^{\mbar \times \mbar},  
$$
and $\mbar:=28-m$.
Let $\mI_m\in\RR^{m \times m}$ and $\mO_{m \times \mbar}\in\RR^{m \times \mbar}$ be the identity matrix and the matrix containing zero elements, respectively.

With the considered ordering of the Hermite degrees of freedom, we are looking for $m$-dimensional subspaces of $\spS_3^2(\Delta_{\WS})$ that can be identified by the first $m$ Hermite degrees of freedom in \eqref{eq:lambda-1}--\eqref{eq:lambda-3}. We assume that these subspaces can be spanned by a vector of (Hermite) basis functions 
\begin{equation}
\label{eq:Hbasism}
\basisHmm{m} := [\Hmm{m}_1,\ldots,\Hmm{m}_m],
\end{equation} 	
where
\begin{equation}
\label{eq:matrix-basis-Hm}
\basisHmm{m} = \basisH \cdot \matRHmm{m}, \quad
\matRHmm{m} := \begin{bmatrix} 
\mI_m \\[0.1cm] \matRHmm{m}_{21} 
\end{bmatrix}.
\end{equation}
Note that $m$ and $\matRHmm{m}_{21}$ have to be properly selected in order to ensure that the resulting subspace contains the space of cubic polynomials. 

Let $\langle \basis{X} \rangle$ denote the space spanned by a given set of functions $\basis{X}$. 
For an efficient local representation of the considered $m$-dimensional subspace $\langle \basisHmm{m}\rangle$, we want to find another vector of basis functions
\begin{equation}
\label{eq:Bbasism}
\basisBmm{m} := [\Bmm{m}_1,\ldots,\Bmm{m}_m],
\end{equation} 
where
\begin{equation}
\label{eq:matrix-basis-Bm}
\basisBmm{m} = \basisB \cdot \matRBmm{m}, \quad
\matRBmm{m} := \begin{bmatrix} 
\mI_m \\[0.1cm] \matRBmm{m}_{21} 
\end{bmatrix},
\end{equation}
such that
\begin{equation*}
\langle \basisHmm{m}\rangle=\langle \basisBmm{m}\rangle.
\end{equation*}
We have the following result.
\begin{proposition}
\label{prop:conversion}
For $m\in\{3,9,18,21,27\}$ the functions in \eqref{eq:Hbasism} and \eqref{eq:Bbasism} span the same subspace of $\spS_3^2(\Delta_{\WS})$ if
\begin{equation}
\label{eq:subspace_condition}
\matRHmm{m}_{21}\matCmm{m}_{11} = \matCmm{m}_{22}\matRBmm{m}_{21}+\matCmm{m}_{21},
\end{equation}
where
$$\matC =: \begin{bmatrix} 
\matCmm{m}_{11} & \matCmm{m}_{12} \\[0.1cm] 
\matCmm{m}_{21} & \matCmm{m}_{22}
\end{bmatrix}, \quad 
\matCmm{m}_{11}\in \RR^{m \times m}, \quad \matCmm{m}_{22}\in \RR^{\mbar \times\mbar}, $$
and $\matC$ is the collocation matrix in \eqref{eq:HB}.
\end{proposition}
\begin{proof}
For notational convenience let us consider the following bases of $\spS_3^2(\Delta_{\WS})$:
\begin{align*}
\basisBme{m} &:= [\Bmm{m}_1,\ldots,\Bmm{m}_m,B_{m+1},\ldots,B_{28}], \\
\basisHme{m} &:= [\Hmm{m}_1,\ldots,\Hmm{m}_m,H_{m+1},\ldots,H_{28}].
\end{align*}
Therefore, setting
$$\matRBme{m} := \begin{bmatrix} 
\mI_m & \mO_{m,\mbar} \\[0.1cm] 
\matRBmm{m}_{21} & \mI_{\mbar}
\end{bmatrix}, \quad 
\matRHme{m} := \begin{bmatrix} 
\mI_m & \mO_{m,\mbar} \\[0.1cm] 
\matRHmm{m}_{21} & \mI_{\mbar}
\end{bmatrix}, $$
we get
$$
\basisBme{m}=\basisB \cdot\matRBme{m}=\basisH \cdot\matC\matRBme{m}
=\basisHme{m} \cdot(\matRHme{m})^{-1}\matC\matRBme{m}.
$$
From Table~\ref{tab:hermiteB} we observe that, for $m\in\{3,9,18,21,27\}$, the matrix $\matC$ has the following lower triangular block structure, with nonsingular diagonal blocks:
$$\matC = \begin{bmatrix} 
\matCmm{m}_{11} & \mO_{m, \mbar} \\[0.1cm] 
\matCmm{m}_{21} & \matCmm{m}_{22}
\end{bmatrix}, \quad 
\matCmm{m}_{11}\in \RR^{m \times m}, \quad \matCmm{m}_{22}\in \RR^{\mbar \times\mbar}. $$
Thus,
\begin{align*}
(\matRHme{m})^{-1}\matC\matRBme{m} &=
\begin{bmatrix} 
\mI_m & \mO_{m,\mbar} \\[0.1cm] 
-\matRHmm{m}_{21} & \mI_{\mbar}
\end{bmatrix}
\begin{bmatrix} 
\matCmm{m}_{11} & \mO_{m, \mbar} \\[0.1cm] 
\matCmm{m}_{21} & \matCmm{m}_{22}
\end{bmatrix}
\begin{bmatrix} 
\mI_m & \mO_{m,\mbar} \\[0.1cm] 
\matRBmm{m}_{21} & \mI_{\mbar}
\end{bmatrix} \\ &=
\begin{bmatrix} 
\matCmm{m}_{11} & \mO_{m, \mbar} \\[0.1cm] 
\matXmm{m}_{21} & \matCmm{m}_{22}
\end{bmatrix},
\end{align*}
where
$$
\matXmm{m}_{21}:=-\matRHmm{m}_{21}\matCmm{m}_{11}+\matCmm{m}_{21}+\matCmm{m}_{22}\matRBmm{m}_{21}.
$$
Hence, \eqref{eq:subspace_condition} shows that
$$
\basisBmm{m}=\basisHmm{m} \cdot\matCmm{m}_{11},
$$
i.e., the two vectors of functions $\basisBmm{m}$ and $\basisHmm{m}$ span the same subspace $\spS_3^2(\Delta_{\WS})$. 
\end{proof}
\begin{remark}	
For $m\in\{3,9,18,21,27\}$ the matrix $\matC$ admits a block decomposition as described in the proof of  Proposition~\ref{prop:conversion}.
Since the diagonal blocks of the matrix $\matC$ are nonsingular, condition \eqref{eq:subspace_condition} allows us to determine the submatrix $\matRHmm{m}_{21}$ once $\matRBmm{m}_{21}$ is given and vice versa. In other words, we can compute the basis $\basisHmm{m}$ in \eqref{eq:matrix-basis-Hm} once the basis $\basisBmm{m}$ in \eqref{eq:matrix-basis-Bm} is fixed and vice versa. Of course, one could also use the relation $\basisBmm{m}=\basisHmm{m} \cdot\matCmm{m}_{11}$.
\end{remark}

\subsection{Subspaces of $\spS_3^2(\Delta_{\WS})$ containing cubic polynomials}
\label{sec:pol_reproduction}
Here we derive general conditions on a vector of basis functions over $\Delta_{\WS}$, denoted by
$\basisBmm{m} := [\Bmm{m}_1,\ldots,\Bmm{m}_m]$,
such that
$$
\spP_3\subseteq \langle \basisBmm{m} \rangle \subseteq \spS_3^2(\Delta_{\WS}).
$$
Moreover, we would like to preserve the connection with the corresponding simplex spline basis $\basisB$ as described in Section~\ref{sec:hermite-simplex}.
Therefore, we consider the functions
 \begin{equation}
 \label{eq:Bmj}
 \Bmm{m}_j:=B_j+\sum_{i=m+1}^{28} r_{ij}^m B_i, \quad r_{ij}^m\in \RR, \quad j=1,\ldots, m,
\end{equation}
and the coefficients $r_{ij}^m$ form the entries of the matrix $\matRBmm{m}_{21} \in \RR^{\mbar \times m}$ in \eqref{eq:matrix-basis-Bm}.
In this perspective, we recall from \cite{LycheMS22} that the basis functions in $\basisB$ enjoy a Marsden-like identity
\begin{equation*}
(1+\vy^T\vx)^3=\sum_{j=1}^{28}\psi_j(\vy)B_j(\vx), \quad \vy\in\RR^2, \quad \vx\in\Delta,
\end{equation*}
where the dual polynomials $\psi_j$, $j=1,\ldots,28$ are defined in \cite[Theorem~4]{LycheMS22}.

\begin{proposition}
\label{prop:pol-reproduction}
The $m$-dimensional subspace of $\spS_3^2(\cT_{\WS})$ spanned by the functions in \eqref{eq:Bmj} contains cubic polynomials if and only if
\begin{equation}
\label{eq:pol-reproduction}
\sum_{j=1}^m\psi_j(\vy) r_{ij}^m=\psi_i(\vy),\quad \vy\in\RR^2, \quad i=m+1,\ldots, 28.
\end{equation}
\end{proposition}
\begin{proof}
To ensure $\spP_3\subseteq \langle \basisBmm{m} \rangle$,
we look for $\psi_j^{m}$, $j=1,\dots, m$ such that
$$
(1+\vy^T\vx)^3=\sum_{j=1}^{m}\psi_j^{m}(\vy)\Bmm{m}_j(\vx), \quad \vy\in\RR^2,\quad \vx\in\Delta.
$$
From \eqref{eq:Bmj} it follows
$$
\sum_{j=1}^{m}\psi_j^{m}(\vy)\Bmm{m}_j(\vx)=
\sum_{j=1}^{m}\psi_j^{m}(\vy)B_j(\vx)+
\sum_{i=m+1}^{28}\sum_{j=1}^m\psi_j^{m}(\vy) r_{ij}^m B_i(\vx).
$$
Since the functions $B_1,\ldots, B_{28}$ form a basis, 
we have
$$
\sum_{j=1}^{28}\psi_j(\vy)B_j(\vx)=\sum_{j=1}^{m}\psi_j^{m}(\vy)\Bmm{m}_j(\vx)
$$
if and only if
$$
\psi_j^{m}(\vy)= \psi_j(\vy),\quad \vy\in\RR^2, \quad j=1,\ldots, m\\
$$
and
$$
\sum_{j=1}^m\psi_j^{m}(\vy) r_{ij}^m=\psi_i(\vy),\quad \vy\in\RR^2, \quad i=m+1,\ldots, 28,
$$
which gives \eqref{eq:pol-reproduction}.
\end{proof}

\subsection{Removing the interior degree of freedom}
\label{sec:interior-dof}
We now detail a specific construction to peel away the Hermite degree of freedom associated with the triangle, i.e., to remove the degree of freedom described in \eqref{eq:lambda-3}. We follow mostly the notation from Section~\ref{sec:pol_reproduction} with $m=27$ to ensure polynomial reproduction.
In other words, we want to construct $27$-dimensional subspaces of $\spS_3^2(\Delta_\WS)$ spanned by basis functions of the form
$$
\Bmm{27}_j:=B_j+ r_j B_{28}, \quad r_j\in \mathbb{R}, \quad j=1,\ldots, 27,
$$
where, to simplify the notation, we set
$$
r_j:=r_{28,j}^{27}, \quad j=1,\ldots,27.
$$
Then, Proposition~\ref{prop:pol-reproduction} requires to look for coefficients $r_j$ such that
\begin{equation}
\label{eq:red27}
\sum_{j=1}^{27}\psi_j(\vy) r_j=\psi_{28}(\vy),\quad \vy\in\RR^2.
\end{equation}
This equation is an equality between cubic polynomials. Therefore, it is satisfied whenever it holds true for ten unisolvent points $\vy_i$, $i=1,\ldots,10$. We choose the standard B\'ezier domain points in a reference triangle $\Delta$. This results in $10$ equations for $27$ unknowns. To reduce the number of possible solutions, it is reasonable to exploit the symmetry of the basis functions $B_1,\ldots, B_{28}$ by collecting them in six ``similarity'' groups and asking that the functions in the same group have the same coefficient. More precisely, we set
\begin{alignat*}{6}
\phi_1(\vy) &:= \sum_{i=1}^3\psi_i(\vy), &\quad
\phi_2(\vy) &:= \sum_{i=4}^9\psi_i(\vy), &\quad
\phi_3(\vy) &:= \sum_{i=10}^{15}\psi_i(\vy),\\
\phi_4(\vy) &:= \sum_{i=16}^{18}\psi_i(\vy), &\quad
\phi_5(\vy) &:= \sum_{i=19}^{21}\psi_i(\vy),&\quad
\phi_6(\vy) &:= \sum_{i=22}^{27}\psi_i(\vy),
\end{alignat*}
and we look for $\widehat r_j$, $j=1,\ldots,6$ such that
\begin{equation}
\label{eq:red27-sim}
\sum_{j=1}^{6}\phi_j(\vy) \widehat r_j=\psi_{28}(\vy),\quad \vy\in\RR^2.
\end{equation}
This implies
$$
r_j:=\begin{cases}
\widehat r_1, & j=1,\ldots,3, \\
\widehat r_2, & j=4,\ldots,9, \\
\widehat r_3, & j=10,\ldots,15, \\
\widehat r_4, & j=16,\ldots,18, \\
\widehat r_5, & j=19,\ldots,21, \\
\widehat r_6, & j=22,\ldots,27. \\
\end{cases}
$$
A direct computation shows that we have a three-parameter family of solutions for \eqref{eq:red27-sim}, namely
\begin{align*}
\widehat r_1 &= 4\widehat r_2+12\widehat r_3+\frac{12}{5}\widehat r_5+\frac{2}{3},\\[0.2cm]
\widehat r_4 &= -15\widehat r_2-35\widehat r_3-7\widehat r_5-2,\\[0.2cm]
\widehat r_6 &= \frac{9}{2}\widehat r_2+\frac{21}{2}\widehat r_3+\frac{9}{5}\widehat r_5+\frac{5}{6}.
\end{align*}
\begin{remark}
An inspection of the family of solutions, and in particular of the expression of $\widehat r_4$, immediately shows that there are no nonnegative solutions $\{\widehat r_j\geq0,\ j=1,\ldots,6\}$. Considering that the support of $B_{28}$ is not contained in any of the supports of $B_j$, $j=1,\dots, 27$, the lack of nonnegative solutions of \eqref{eq:red27-sim} implies that there are no bases $\basisBmm{27}$ consisting of solely nonnegative functions.
\end{remark}
\begin{remark} The minimum value of a component of a solution of \eqref{eq:red27-sim} cannot be greater than $-\frac{1}{29}$. Indeed, if we
assume that 
$$\widehat r_j\geq w>-\frac{1}{29},\quad j=1,\ldots, 6,$$
then we have that 
$\widehat r_4\leq-15w-35w-7w-2=-57w-2<-\frac{1}{29}$. This gives a contradiction.
\end{remark}

\begin{examples}
In this example we report some interesting solutions of \eqref{eq:red27} under the ``similarity'' constraint. We only show the vector $[\widehat r_1,\ldots,\widehat r_6]$.
\begin{itemize}
\item Solution with only nonzero coefficients corresponding to 12 of the 18 basis functions associated with the vertices:
\begin{equation*}
\left[-\frac{2}{27},-\frac{5}{27},0,\frac{7}{9},0,0\right].
\end{equation*}
\item Solution with only nonzero coefficients corresponding to basis functions that are zero on the boundary:
\begin{equation*}
\left[0,0,0,-\frac{1}{18},-\frac{5}{18},\frac{1}{3}\right].
\end{equation*}
\item Solution with a large number of zero coefficients:
\begin{equation}\label{eq:rhat-3}
\left[-\frac{4}{9},0,0,\frac{67}{54},-\frac{25}{54},0\right];
\end{equation}
see Figures~\ref{fig:B27-zeros-a}--\ref{fig:B27-zeros-b} for an illustration of the corresponding basis functions.
\item Solution with least negative coefficients:
\begin{equation}\label{eq:rhat-4}
\left[\frac{14}{435},-\frac{1}{29},-\frac{1}{29},-\frac{1}{29},-\frac{1}{29},\frac{221}{870}\right];
\end{equation}
see Figures~\ref{fig:B27-least-a}--\ref{fig:B27-least-b} for an illustration of the corresponding basis functions.
\end{itemize}
\end{examples}

\begin{figure}
\centering
\includegraphics[width=3.8cm]{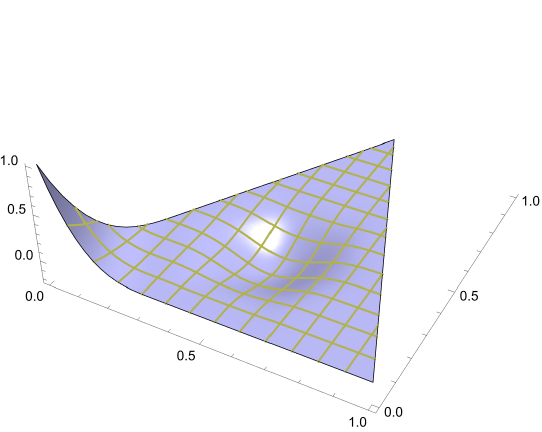}
\includegraphics[width=3.8cm]{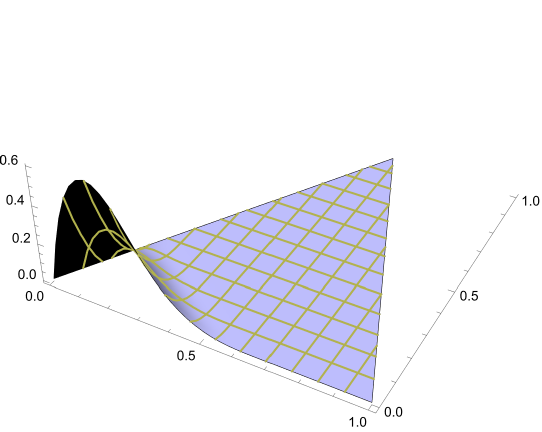}
\includegraphics[width=3.8cm]{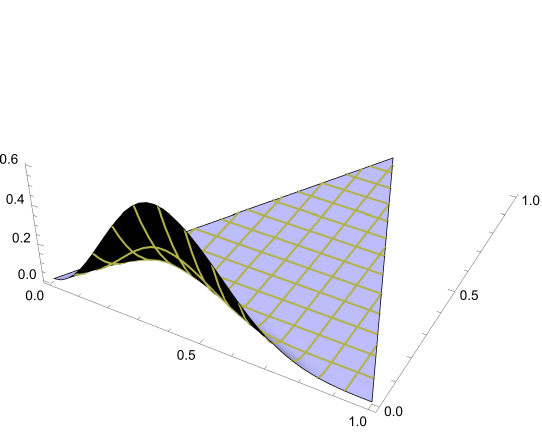}
\caption{The basis functions $B_1^{27}$, $B_4^{27}$, and $B_{10}^{27}$, constructed according to the procedure in Section~\ref{sec:interior-dof} with $[\widehat r_1,\ldots,\widehat r_6]$ in \eqref{eq:rhat-3}. See Figure~\ref{fig:B1-B4-B10} for comparison.}
\label{fig:B27-zeros-a}
\smallskip
\centering
\includegraphics[width=3.8cm]{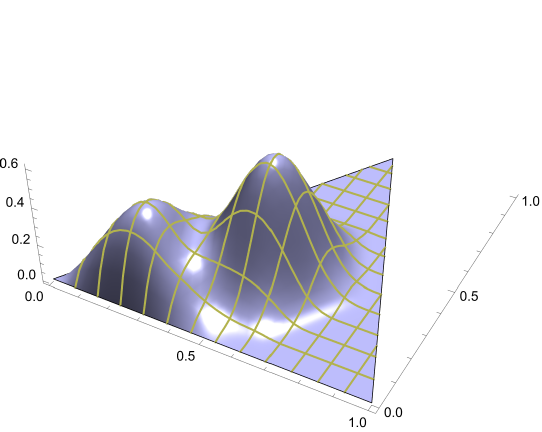}
\includegraphics[width=3.8cm]{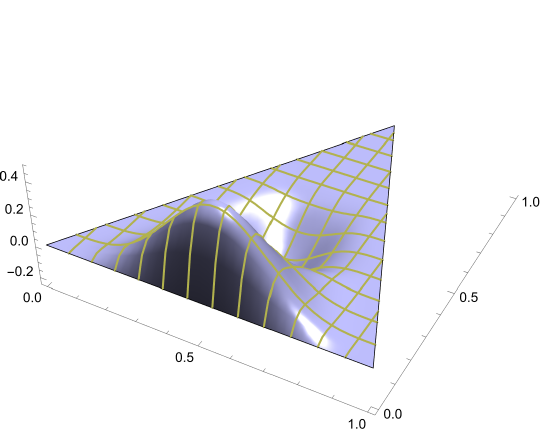}
\includegraphics[width=3.8cm]{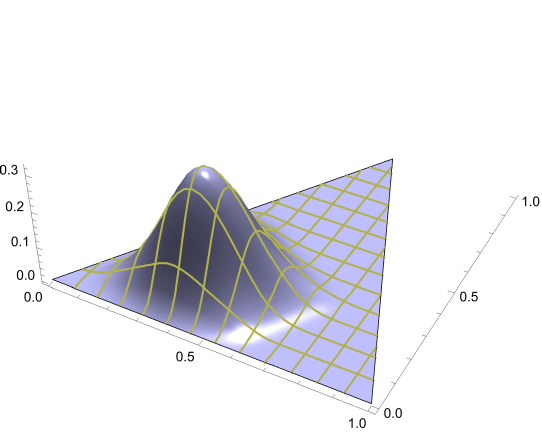}
\caption{The basis functions $B_{16}^{27}$, $B_{19}^{27}$, and $B_{22}^{27}$, constructed according to the procedure in Section~\ref{sec:interior-dof} with $[\widehat r_1,\ldots,\widehat r_6]$ in \eqref{eq:rhat-3}. See Figure~\ref{fig:B16-B19-B22} for comparison.}
\label{fig:B27-zeros-b}
\smallskip
\centering
\includegraphics[width=3.8cm]{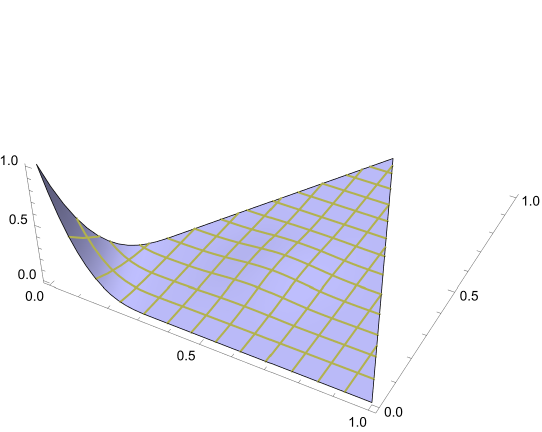}
\includegraphics[width=3.8cm]{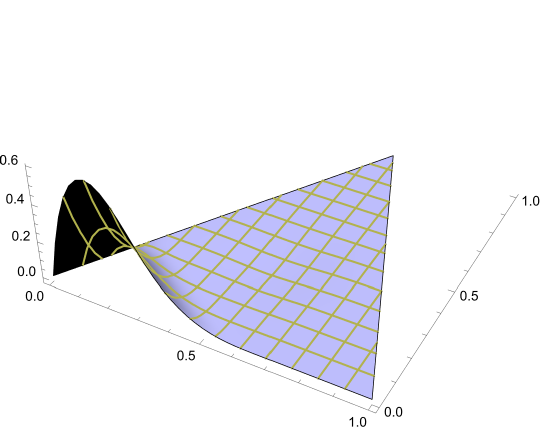}
\includegraphics[width=3.8cm]{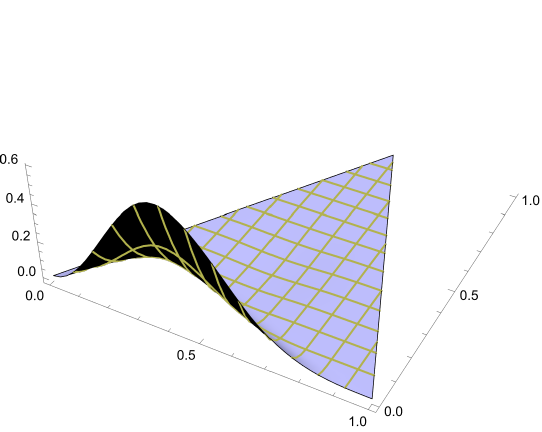}
\caption{The basis functions $B_1^{27}$, $B_4^{27}$, and $B_{10}^{27}$, constructed according to the procedure in Section~\ref{sec:interior-dof} with $[\widehat r_1,\ldots,\widehat r_6]$ in \eqref{eq:rhat-4}. See Figure~\ref{fig:B1-B4-B10} for comparison.}
\label{fig:B27-least-a}
\smallskip
\centering
\includegraphics[width=3.8cm]{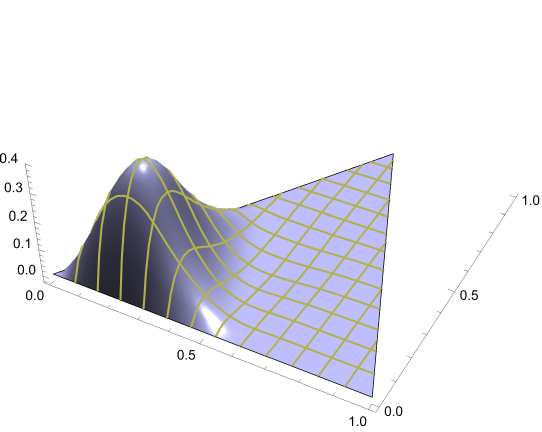}
\includegraphics[width=3.8cm]{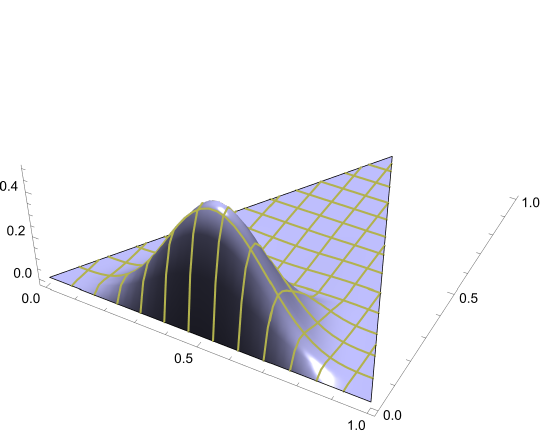}
\includegraphics[width=3.8cm]{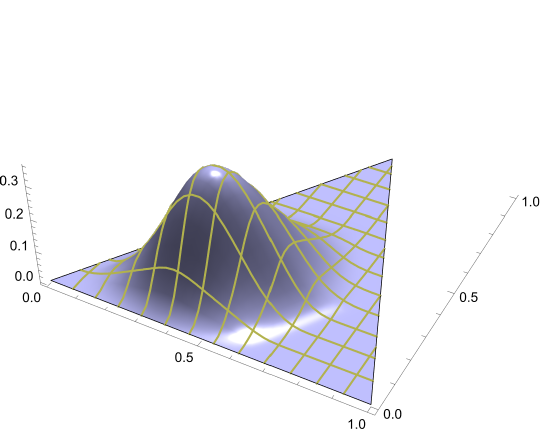}
\caption{The basis functions $B_{16}^{27}$, $B_{19}^{27}$, and $B_{22}^{27}$, constructed according to the procedure in Section~\ref{sec:interior-dof} with $[\widehat r_1,\ldots,\widehat r_6]$ in \eqref{eq:rhat-4}. See Figure~\ref{fig:B16-B19-B22} for comparison.}
\label{fig:B27-least-b}
\end{figure}

\section{Reduced spaces of $\spS_3^2(\cT_\WS)$}
\label{sec:reduced-spaces-global}
In Section~\ref{sec:reduced-spaces-local} we discussed how to construct subspaces of $\spS_3^2(\Delta_\WS)$ that contain cubic polynomials and can be described by Hermite data only associated with the vertices and edges of $\Delta$. We also showed how to represent them in terms of the local simplex spline basis $\basisB$. 
For $m\in\{18,21,27\}$ we now summarize a general method how to apply these local subspace strategies on each triangle $\Delta$ of $\cT$ to obtain global subspaces  
$$\spS_3^{2,m}(\cT_\WS)\subset \spS_3^2(\cT_\WS)$$ 
with the following properties:
\begin{itemize}
\item $\dim( \spS_3^{2,m}(\cT_\WS))=6n_V+\begin{cases}
0 &m=18,
\\
n_E& m=21,
\\
3n_E& m=27;
\end{cases}$
\item $\spP_3\subset \spS_3^{2,m}(\cT_\WS)$;
\item any element of $\spS_3^{2,m}(\cT_\WS)$ is uniquely identified by the different Hermite interpolation conditions $\lambda_j$ in \eqref{eq:lambda-1}--\eqref{eq:lambda-3} for $j=1,\ldots, m$ related to all triangles $\Delta$ of $\cT$;
\item any element of $\spS_3^{2,m}(\cT_\WS)$ can be locally constructed on each triangle $\Delta$ of $\cT$ separately; 
\item $\spS_3^{2,m}(\cT_\WS)$ is equipped with a local simplex spline representation.
\end{itemize}

Our goal can be reached by patching together local subspaces of dimension $m$, $\spS_3^{2,m}(\Delta_\WS)\subset \spS_3^{2}(\Delta_\WS)$, built as described in the procedure below. These local subspaces are equipped with a simplex spline representation and admit a simple global $C^2$ coupling. A characterization in terms of the first $m$ Hermite degrees of freedom in \eqref{eq:lambda-1}--\eqref{eq:lambda-3} follows also from their construction.
\begin{procedure}
\label{proc:assembly}
For $m\in\{18,21,27\}$ and each $\Delta\in\cT$, the subspace $\spS_3^{2,m}(\Delta_\WS)\subset \spS_3^{2}(\Delta_\WS)$ is constructed as follows.
\begin{enumerate}
\item  If $m=27$: 
\begin{enumerate}
\item construct the basis $\basisBmm{27}$ according to the procedure in Section~\ref{sec:interior-dof};
\item set
$$ \spS_3^{2,m}(\Delta_\WS):=\langle\basisBmm{27}\rangle.$$		
\end{enumerate}
\item If $m=18$ or $m=21$: 
\begin{enumerate}
\item construct the basis $\basisBmm{27}$ according to the procedure in Section~\ref{sec:interior-dof} and compute the corresponding matrix $\matRBmm{27}_{21}$;
\item determine the Hermite basis $\basisHmm{27}$ such that $\langle\basisHmm{27}\rangle=\langle\basisBmm{27}\rangle$
by using the form \eqref{eq:matrix-basis-Hm} with $\matRHmm{27}_{21}$ obtained from \eqref{eq:subspace_condition};
\item remove the Hermite degrees of freedom connected to $\lambda_j$, $j=m+1,\ldots, 27$, according to the strategies described in Section~\ref{sec:peeling} and compute the corresponding matrix $\matRHmm{m}_{21}$;
\item determine the basis $\basisBmm{m}$ such that $\langle\basisHmm{m}\rangle=\langle\basisBmm{m}\rangle$
by using the form \eqref{eq:matrix-basis-Bm} with $\matRBmm{m}_{21}$ obtained from \eqref{eq:subspace_condition};
\item set 
$$ \spS_3^{2,m}(\Delta_\WS):=\langle\basisBmm{m}\rangle.$$
\end{enumerate}
\end{enumerate}
\end{procedure}

Given the local subspaces $\spS_3^{2,m}(\Delta_\WS)\subset \spS_3^{2}(\Delta_\WS)$ constructed according to Procedure~\ref{proc:assembly}, where the same (reduced) Hermite interpolation conditions across the edges of $\cT$ are considered in Step 2.c., we define the global subspace
$\spS_3^{2,m}(\cT_\WS)\subset \spS_3^{2}(\cT_\WS)$ as
\begin{equation}\label{eq:def_space_reduced}
\spS_3^{2,m}(\cT_\WS):=\{\spline\in C^2(\Omega): \spline_{|\Delta}\in \spS_3^{2,m}(\Delta_\WS), \ \Delta\in \cT \}.
\end{equation}
It can be easily checked that the space in \eqref{eq:def_space_reduced} enjoys the properties listed in the beginning of the section.
Note that the use of the same (reduced) Hermite interpolation conditions across the edges of  $\cT$, as discussed in Section~\ref{sec:peeling}, ensures that the local Hermite basis functions of the subspaces $\spS_3^{2,m}(\Delta_\WS)$ can be immediately joined to form globally $C^2$ functions, in the spirit of classical macro-elements.
As for the peeling of the interior degree of freedom, we can directly work with the simplex spline basis, without involving the Hermite basis. Indeed, the procedure in Section~\ref{sec:interior-dof} only redistributes the contribution of the function $B_{28}$ among the  remaining basis functions $B_1,\ldots, B_{27}$. Since $B_{28}$ is of class $C^2$ across the edges of $\Delta$ this does not affect the $C^2$ smoothness across these edges
of the functions in the subspace. This explains the difference between the cases $m=27$ and $m\in\{18,21\}$ in Procedure~\ref{proc:assembly}.

\begin{remark}
Due to the structure of the local bases $\basisBmm{m}$ obtained from Procedure~\ref{proc:assembly}, the corresponding local representations give rise to $C^2$ smoothness conditions across edges of $\cT$ that are similar to the ones derived in \cite[Theorems~5 and~6]{LycheMS22}.
\end{remark}

\begin{remark}
The approach presented in Procedure~\ref{proc:assembly} could also be applied to remove some degrees of freedom associated with the vertices: second derivatives ($m=9$) or first and second derivatives ($m=3$); see Proposition~\ref{prop:conversion}. However, this is of no practical interest because it would result in subspaces that do not contain cubic polynomials and so only possess a reduced approximation power.  
\end{remark}

\section{Conclusion}
\label{sec:conclusion}
We have presented a framework for constructing suitable subspaces of the $C^2$ cubic spline space defined on the $\WS$ refinement of a given triangulation $\cT$.
Adopting a procedure common in classical finite element constructions, the subspaces are characterized by specific Hermite degrees of freedom associated with only the vertices and edges of $\cT$, or even only the vertices of $\cT$. The $\WS$ split of each triangle provides a sufficiently large number of degrees of freedom. Therefore, contrary to some reduction strategies known in the literature for popular finite elements (such as the Bell element \cite{Bell.69}), the presented reduction still contains cubic polynomials, a necessary condition for attaining full approximation order.

The characterization in terms of (reduced) Hermite degrees of freedom allows us to locally construct any spline on each triangle of $\cT$ separately, in the spirit of classical macro-elements.
We have also provided the representation of the obtained local subspaces in terms of a local simplex spline basis. This avoids the need to consider separate polynomial representations on each of the polygonal subregions of the $\WS$ split, making the complex geometry of the $\WS$ split transparent to the user.

The minimum number of degrees of freedom necessary for a local and symmetric formulation of the proposed strategy is six times the number of vertices of the given triangulation. This considerably reduces the dimension of the space of $C^2$ cubic splines on the $\WS$ refined triangulation; see \eqref{eq:dim-space}.

In practice, the implementation of the reduction strategy outlined in Section~\ref{sec:reduced-spaces-global} requires the knowledge of the conversion matrix between the Hermite and the simplex spline basis for the full space on each macro-triangle; see Table~\ref{tab:hermiteB} for its sparsity structure. The exact values of the entries of such a matrix can be obtained by standard elementary calculus from the values in \cite[Tables~3--5]{LycheMS22}. It is clear that these values depend on the specific geometry of each triangle and their explicit expressions might be cumbersome in terms of general vertex coordinates, so we did not explicitly report them here for a general triangle. We just refer the reader to Table~\ref{tab:hermiteB2-reference-tri-h} for the case of the scaled unit triangle.

Besides polynomial reproduction, a formal proof of full approximation order of the reduced spaces requires showing that they are equipped with stable locally supported bases; see \cite[Chapters~5 and~10]{Lai.Schumaker07}. On each refined triangle $\Delta_\WS$, the adopted local simplex spline basis behaves like the classical Bernstein polynomial basis on a triangle. In particular, it enjoys local smoothness conditions across the edges of the given triangulation $\cT$, completely analogous to those well known for Bernstein--B\'ezier representations; see, e.g., \cite[Chapter~2]{Lai.Schumaker07} or \cite[Section~3]{MS-this-volume} for an introduction. In this perspective, the use of the local simplex spline basis allows for an extension of the minimal determining set method \cite{Lai.Schumaker07} to construct stable locally supported bases, simply by replacing the Bernstein polynomial basis by the local simplex spline basis; see \cite[Section~4]{LycheMS22} for more details. This is an additional reason to prefer the local bases $\basisB$ and $\basisB^m$ when representing the full and the reduced spaces we are dealing with.

\section*{Acknowledgements}
This work was supported 
by the MUR Excellence Department Project \text{MatMod@TOV} (CUP E83C23000330006) awarded to the Department of Mathematics of the University of Rome Tor Vergata
and by the National Research Center in High Performance Computing, Big Data and Quantum Computing (CUP E83C22003230001).
C. Manni and H. Speleers are members of Gruppo Nazionale per il Calcolo Scientifico, Istituto Nazionale di Alta Matematica.

\end{document}